\RequirePackage{fix-cm}
\documentclass[smallextended,final]{svjour3}
\smartqed

\usepackage{geometry}
\usepackage{graphicx,color}
\graphicspath{{figures/}{../figures/}{../../figures/}}
\usepackage[pdfpagelabels]{hyperref}
\usepackage{subfiles}
\usepackage{amsmath,amssymb}
\usepackage{framed}
\usepackage{enumerate}
\usepackage{ascmac}
\usepackage{blindtext}
\usepackage{url}
\usepackage{cases}
\usepackage{lscape}

\spnewtheorem{assumption}{Assumption}{\bf}{\it}

\DeclareMathOperator{\interior}{int}
\DeclareMathOperator{\dom}{dom}
\DeclareMathOperator{\dist}{dist}

\DeclareMathOperator{\real}{\mathbb{R}}
\DeclareMathOperator*{\argmax}{argmax}
\DeclareMathOperator*{\argmin}{argmin}

\newcommand{\ie}{\textit{i.e.}}
\newcommand{\etal}{\textit{et al.}}

\usepackage{algorithm,algpseudocode}

\algnewcommand{\Input}[1]{
\State\textbf{Input:}\hspace*{0.5em}\parbox[t]{.9\linewidth}{\raggedright #1}
}
\algnewcommand{\Initialization}[1]{
\State\textbf{Initialization:}\hspace*{0.5em}\parbox[t]{.8\linewidth}{\raggedright #1}
}

\begin{document}
    \title{New Bregman proximal type algorithms for solving DC optimization problems}
    \author{Shota~Takahashi \and Mituhiro~Fukuda \and Mirai~Tanaka}
    \institute{
        Shota Takahashi
        \at Department of Statistical Science,
        The Graduate University for Advanced Studies,
        10--3 Midori-cho,
        Tachikawa, Tokyo 190--8562, Japan.\\
        \email{takahashi.shota@ism.ac.jp}\\
        \and Mituhiro Fukuda
        \at Department of Computer Science,
        Institute of Mathematics and Statistics,
        University of S\~ao Paulo,
        Rua do Mat\~ao, 1010, Cidade Universit\'aria, S\~ao Paulo 05508--090, Brazil;
        Department of Mathematical and Computing Science,
        Tokyo Institute of Technology,
        2--12--1 Oh-okayama,
        Meguro-ku, Tokyo 152--8552, Japan;
        currently at S\~ao Paulo State Technological College, Praia
        Grande, Pra\c{c}a 19 de Janeiro, 144, Praia
        Grande, SP, 11700--100, Brazil.
        \email{mituhiro@is.titech.ac.jp}\\
        \and Mirai Tanaka
        \at Department of Statistical Inference and Mathematics,
        The Institute of Statistical Mathematics,
        10--3 Midori-cho,
        Tachikawa, Tokyo 190--8562, Japan;
        Continuous Optimization Team, RIKEN Center for Advanced Intelligence Project,
        Nihonbashi 1--chome Mitsui Building, 15th floor, 1--4--1 Nihonbashi, Chuo-ku, 103--0027 Tokyo, Japan.\\
        \email{mirai@ism.ac.jp}
    }
    
    \date{Received: \today / Accepted: date}

    \maketitle

    \begin{abstract}
        Difference of Convex (DC) optimization problems have objective functions that are differences between two convex functions.
        Representative ways of solving these problems are the proximal DC algorithms, which require that the convex part of the objective function have $L$-smoothness.
        In this article, we propose the Bregman Proximal DC Algorithm (BPDCA)
        for solving large-scale DC optimization problems that do not possess $L$-smoothness.
        Instead, it requires that the convex part of the objective function has
        the $L$-smooth adaptable property that is exploited in Bregman proximal gradient algorithms.
        In addition, we propose an accelerated version, the Bregman Proximal DC Algorithm with extrapolation (BPDCAe), with a new restart scheme.
        We show the global convergence of the iterates generated by BPDCA(e) to a limiting critical point
        under the assumption of the Kurdyka-\L ojasiewicz property or subanalyticity of the objective function
        and other weaker conditions than those of the existing methods.
        We applied our algorithms to phase retrieval,
        which can be described both as a nonconvex optimization problem and as a DC optimization problem.
        Numerical experiments showed that BPDCAe outperformed existing Bregman proximal-type algorithms
        because the DC formulation allows for larger admissible step sizes.
    \end{abstract}



    \section{Introduction}\label{sec:introduction}
    We are interested in solving Difference of Convex (DC) optimization problems:
    \begin{align*}
    (\mathcal{P}) \quad \min \left\{ \Psi(x) := f_1(x) - f_2(x) + g(x) \ \middle| \ x \in \overline{C} \right\},
    \end{align*}
    where $f_1$ and $f_2$ are convex functions on $\real^d$,
    and $\overline{C}$ is the closure of $C$ which is a nonempty, convex, and open set.
    Note that the function $f_1 - f_2$ is not always convex.
    Also, $g$ may be nonsmooth, such as the $\ell_1$-norm $\|x\|_1$ in~\cite{bpg,Mukkamala2020,bpge}, or alternatively,
    $f_2$ may be nonsmooth \cite{hoai2018}.
    Some interesting examples of $(\mathcal{P})$ can be found in~\cite{pdcae}.
    Although we will place some assumptions on $C$, it can be regarded as $\real^d$ for simplicity.
    Applications of DC optimization are summarized in~\cite{Horst_1999,lethi18,Tuy_1995}.
    \nocite{Horst_1999}\nocite{lethi18}\nocite{Tuy_1995}

    The DC Algorithm (DCA) (see for instance~\cite{lethi18}) is a well-known iterative method for solving the DC optimization problems $(\mathcal{P})$.
    At each iteration, its computational burden mainly depends on the resolution of the subproblem,
    \begin{align}
        \label{subprob:dca}
        x^{k+1} = \argmin_{x \in \overline{C}} \left\{ f_1(x) - \langle \xi^k, x - x^k \rangle + g(x) \right\},
    \end{align}
    where $\xi^k \in \partial_{\mathrm{c}} f_2(x^k) := \{\xi\in\real^d\ |\ f(y) - f(x^k) - \langle \xi, y - x^k \rangle \geq 0, \forall y \in\real^d\}$ is a (classical) subgradient of $f_2$ at $x^k\in\overline{C}$.
    Solving subproblem~\eqref{subprob:dca} may be computationally demanding unless $f_1$ and $g$ have simple structure or $(\mathcal{P})$ is small-scale.
    When $g$ is convex, the proximal DC Algorithm (pDCA) (see for instance~\cite{pdcae}) is an alternative method of solving large-scale DC optimization problems.
    However, to guarantee global convergence of its iterates to a critical point, $f_1$ needs to be $L$-smooth; \ie, its gradient needs to be globally Lipschitz continuous.
    Each step of pDCA is given by
    \begin{align}
        \label{subprob:pdca}
        &x^{k+1} = \argmin_{x \in \overline{C}} \left\{ g(x) + \langle \nabla f_1(x^k) - \xi^k, x - x^k \rangle + \frac{1}{2\lambda}\|x - x^k\|^2 \right\},
    \end{align}
    where $\xi^k \in \partial_{\mathrm{c}} f_2(x^k)$, $x^k\in\overline{C}$,
    $\lambda > 0$ satisfies $0 < \lambda L < 1$, and $\|\cdot\|$ denotes the Euclidean norm.
    Since $\lambda\ (< 1/L)$ plays the role of a step size,
    finding a larger upper bound $1/L$, \ie, finding a smaller $L$,
    is of fundamental importance to achieving fast convergence.
    Wen \etal~\cite{pdcae} proposed the proximal DC Algorithm with extrapolation (pDCAe)
    to accelerate pDCA with the extrapolation technique, which is used,
    for instance, in the Fast Iterative Shrinkage-Thresholding Algorithm (FISTA) of Beck and Teboulle~\cite{beck09}
    and in the Nesterov\rq{}s extrapolation technique~\cite{nesterov83,Nesterov2018}.

    Bolte \etal~\cite{bpg}, who incorporated the kernel generating distance (function) $h$ and Bregman distances \cite{bregman} into the algorithm framework, came up with the notion of the $L$-smooth adaptable property (see also \cite{Bauschke_2017}).
    This property is less restrictive than $L$-smoothness.
    A variant of the Bregman Proximal Gradient algorithm (BPG) proposed by Mukkamala \etal~\cite{Mukkamala2020} iteratively estimates a small $L$,
    while Zhang \etal~\cite{bpge} proposed the Bregman Proximal Gradient algorithm with extrapolation (BPGe),
    which combines BPG with a line search step for extrapolating parameters.

    In this paper, we propose two new algorithms, namely, the Bregman Proximal Difference of Convex Algorithm (BPDCA) and
    the Bregman Proximal Difference of Convex Algorithm with extrapolation (BPDCAe), which are inspired by pDCA(e) and BPG(e).
    These novel algorithms combine pDCA(e) with the Bregman distances.
    In the subproblem of BPDCA(e), the use of Bregman distances guarantees the accuracy of a linear approximation of $f_1-f_2$.

    The novelty of our contributions can be better understood by comparing them with the existing work.
    As already mentioned, Bregman distances allow us to extend the class of functions to be minimized $f_1$
    from $L$-smooth in pDCAe \cite{pdcae} to the larger class of $L$-smooth adaptable pairs of functions $(f_1,h)$.
    In addition, the function $g$ does not need to be convex in the case of BPDCA.
    By assuming that either $f_2$ or $g$ are differentiable and that their gradients are locally Lipschitz continuous,
    the iterates of BPDCA(e) converge globally to a limiting stationary point (Theorems~\ref{theorem:global-convergence-bpdca} and \ref{theorem:global-convergence-bpdcae}) or a limiting critical point (Theorems~\ref{theorem:global-convergence-bpdca-g} and~\ref{theorem:global-convergence-bpdcae-g}),
    where the definitions of these convergent points are given in Definition~\ref{def:critical-stationary}.
    This means that either $g$ or $f_2$ can be nonsmooth.
    
    Compared with BPG-type algorithms~\cite{bpg,Mukkamala2020,bpge},
    BPDCA(e) well exploits the structure of the objective function.
    When applying these BPG-type algorithms to solve problem $(\mathcal{P})$, we decompose $\Psi$ into two functions.
    There are two naive ways to decompose $\Psi$.
    First, we consider the decomposition $\Psi = f_1 + (g - f_2)$ to apply BPG.
    In this case, BPG solves its subproblem $\min\{g(x) - f_2(x) + \langle\nabla f_1(x^k), x-x^k\rangle+\frac{1}{\lambda}D_h(x, x^k)\}$ at the $k$th iteration,
    where $D_h$ is the Bregman distance associated with a kernel generating distance $h$ (see Definition~\ref{def:kernel}) and $\lambda < 1/L$ is a positive parameter.
    In general, it is difficult to efficiently solve it because $f_2$ often does not have simple structure such as separability.
    This fact is also true even if $g$ is convex and separable, simultaneously.
    With BPDCA, we only need to solve its subproblem $\min\{g(x) + \langle\nabla f_1(x^k) - \xi^k, x-x^k\rangle+\frac{1}{\lambda}D_{h'}(x, x^k)\}$, where $\xi^k\in\partial_{\mathrm{c}}f_2(x^k)$.
    If $g$ is additionally convex, the subproblem becomes convex and hence is often efficiently solved.
    Moreover, if $g$ and $h$ are also separable, the subproblem is reduced to $d$ independent one-dimensional convex optimization problems. 
    Even without separability of $h$, it often has closed-form solution formulae as we mentioned in Section~\ref{sec:applications}.
    As an alternative way of decomposition of $\Psi$, we consider $\Psi = (f_1 - f_2) + g$ to apply BPG.
    In this case, to guarantee global convergence, the $L$-smooth adaptability of $(f_1-f_2,h)$ is required (see Definition~\ref{def:l-smad}).
    Meanwhile, for the global convergence of BPDCA(e), the $L'$-smooth adaptability of $(f_1,h')$ is required.
    Comparing these constants, $L'\leq L$ in general,
    and then, we can expect substantial decrease in the objective function at each iteration (Lemma~\ref{lemma:obj-decrease} and \cite[Lemma 4.1]{bpg}).
    This fact has dramatic consequences in practice, as we found in numerical experiments on phrase retrieval (Subsection~\ref{subsec:application-to-quadratic-inverse-problems}).

    The convergence of our algorithms and the
    monotonicity of the objective function are based on standard assumptions.
    Our new restart scheme (Subsection~\ref{subsec:bregman-proximal-dc-algorithm-with-extrapolation}) plays an important role
    in guaranteeing the non-increasing property of the objective functions of BPDCAe without the need for a line search, as in~\cite{bpge}.
    We show global convergence under local Lipschitz continuity of the gradients and
    the Kurdyka-\L ojasiewicz property or subanalyticity of the objective function.
    Additionally, we evaluated the rates of convergence of BPDCA(e).

    To evaluate the performance of BPDCA(e),
    we applied them to phase retrieval,
    a well-known application in nonconvex optimization.
    Phase retrieval arises in many fields of science and engineering,
    such as X-ray crystallography and image processing~\cite{wirtinger,pr-optical}.\nocite{pr-optical}
    It can be formulated as a nonconvex optimization problem or DC optimization problem $(\mathcal{P})$, such as in~\cite{Huang2018}.
    It cannot be solved via pDCA or proximal algorithms, since the function we want to minimize is not $L$-smooth.
    When we formulated phase retrieval as a DC optimization problem,
    we obtained much smaller $L$-smooth adaptable parameters than the existing ones~\cite[Lemma 5.1]{bpg}, \cite{Mukkamala2020}.
    Thus, our algorithms outperformed BPG(e) in our numerical experiment.
    Further experiments showed that, under the Gaussian model, BPDCAe had a higher success rate of
    phrase retrieval than that of Wirtinger flow \cite{wirtinger}.
    Although the kernel generating distance $h$ we utilized does not satisfy Assumption~\ref{ass4} (i),
    the sequences generated by BPDCA(e) converged in the numerical experiments.
    Therefore, we conjecture that all convergence analyses can be carried out with weaker conditions.

    This paper is organized as follows.
    Section~\ref{sec:preliminaries} summarizes notions such as the limiting subdifferential, the Bregman distances, the $L$-smooth adaptable property, the Kurdyka-\L ojasiewicz property, and the subanalytic functions.
    Section~\ref{sec:algorithms-with-bregman-distance} introduces our Bregman proximal-type algorithms under the assumption that $(f_1,h)$ has an $L$-smooth adaptable property.
    Section~\ref{sec:convergence-analysis} (and Appendix~\ref{sec:appendix}) establishes the global convergence of BPDCA(e) to a limiting stationary point or a limiting critical point of the problem $(\mathcal{P})$ and analyzes its rate of convergence.
    Section~\ref{sec:applications} derives small values for the constant $L$ and compares the performance of our algorithms with that of BPG(e).
    Section~\ref{sec:conclusions} summarizes our contributions and discusses future work.

    \section{Preliminaries}\label{sec:preliminaries}
    Here, we review the important notions we will need in the subsequent sections.
    \subsection{Subdifferentials}\label{subsec:subdifferential}
    \begin{definition}[Limiting Subdifferential~\cite{varAna}]
        \label{def:limiting-subdiff}
        For a proper and lower semicontinuous function $f:\real^d\to (-\infty, +\infty]$,
        the limiting subdifferential~\cite{varAna} of $f$ at $x\in\dom f$ is defined by
        \begin{align*}
            \partial f(x) = \left\{\xi\in\real^d \ \middle|\ \exists x^k \xrightarrow{f} x,\xi^k\to\xi\
            \mathrm{such\ that}\ \liminf_{y\to x^k, y\neq x^k}\frac{f(y) - f(x^k) - \langle \xi^k, y - x^k \rangle}{\|y - x^k\|}\geq 0\right\},
        \end{align*}
        where $x^k \xrightarrow{f} x$ means $x^k \to x$ and $f(x^k)\to f(x)$.
    \end{definition}
    Note that when $f$ is convex, the limiting subdifferential coincides with the (classical) subdifferential~\cite[Proposition 8.12]{varAna}, that is,
    $\partial f(x) = \partial_{\mathrm{c}} f(x)$ for all $x\in\real^d$.

    \subsection{Bregman Distances}\label{subsec:proximity-measures}
    First, we define kernel generating distances and Bregman distances.
    \begin{definition}[Kernel Generating Distances~\cite{bpg} and Bregman Distances~\cite{bregman}]
        \label{def:kernel}
        Let $C$ be a nonempty, convex, and open subset of $\real^d$.
        Associated with $C$, a function $h:\real^d\to(-\infty, +\infty]$ is called a kernel generating distance if it meets the following conditions:
        \begin{enumerate}
            \renewcommand{\labelenumi}{\rm{(\roman{enumi})}}
            \item $h$ is proper, lower semicontinuous, and convex, with $\dom h \subset \overline{C}$ and $\dom \partial h = C$.
            \item $h$ is $\mathcal{C}^1$ on $\interior\dom h = C$.
        \end{enumerate}
        We denote the class of kernel generating distances by $\mathcal{G}(C)$.
        Given $h \in \mathcal{G}(C)$, the Bregman distance $D_h: \dom h \times \interior\dom h \to \real_+$ is defined by
        \begin{align*}
            D_h(x, y) := h(x) - h(y) - \langle \nabla h(y), x - y \rangle.
        \end{align*}
    \end{definition}

    From the gradient inequality, $h$ is convex if and only if
    $D_h(x, y) \geq 0$ for any $x \in \dom h $ and $y \in \interior\dom h$.
    When $h$ is a strictly convex function, the equality holds if and only if $x = y$.
    When $h = \frac{1}{2} \|\cdot\|^2$, $D_h(x, y) = \frac{1}{2} \|x - y\|^2$,
    which is the squared Euclidean distance.

    In addition, the Bregman distances satisfy the three-point identity~\cite{bpg},
    \begin{align}
        \label{eq:three-point}
        D_h(x, z) - D_h(x, y) - D_h(y, z) = \langle \nabla h(y) - \nabla h(z), x - y \rangle,
    \end{align}
    for any $y, z \in \interior\dom h$, and $x \in \dom h$.

    \subsection{Smooth Adaptable Functions}\label{subsec:smooth-adaptable-function}
    Now let us define the notion of $L$-smooth adaptable.
    \begin{definition}[$L$-smooth adaptable~\cite{bpg}]
        \label{def:l-smad}
        Consider a pair of functions $(f, h)$ satisfying the following conditions:
        \begin{enumerate}
            \renewcommand{\labelenumi}{\rm{(\roman{enumi})}}
            \item $h \in \mathcal{G}(C)$,
            \item $f: \real^d \to (-\infty, +\infty]$ is a proper and lower semicontinuous function with
            $\dom h \subset \dom f$, which is $\mathcal{C}^1$ on $C = \interior\dom h$.
        \end{enumerate}
        The pair $(f, h)$ is called $L$-smooth adaptable ($L$-smad) on $C$ if there exists $L > 0$ such that $Lh - f$ and $Lh + f$ are convex on $C$.
    \end{definition}

    Since our focus is on DC optimization,
    the function $f_1$ in $(\mathcal{P})$ is always convex.
    Thus, it will be enough to verify that $Lh-f_1$ is convex on $C$ to have $(f_1,h)$ $L$-smad on $C$.

    From the $L$-smooth adaptable property comes the Descent Lemma~\cite{bpg}.
    \begin{lemma}[Full Extended Descent Lemma~\cite{bpg}]\label{lemma:descent-lemma}
    A pair of functions $(f, h)$ is $L$-smad on $C=\interior\dom h$ if and only if:
    \begin{align*}
        |f(x) - f(y) - \langle \nabla f(y), x - y \rangle | \leq L D_h(x, y),\quad \forall x, y \in \interior\dom h.
    \end{align*}
    \end{lemma}

    \subsection{Kurdyka-\L ojasiewicz Property and Subanalytic Functions}\label{subsec:kurdyka-lojasiewicz-property}
    Given $\eta>0$, let $\Xi_{\eta}$ denote the set of all continuous concave functions $\phi:[0, \eta) \to \real_+$
    that are $\mathcal{C}^1$ on $(0, \eta)$ with positive derivatives and which satisfy $\phi(0) = 0$.
    Here, we introduce the Kurdyka-\L ojasiewicz property~\cite{palm,kl}, which we need when analyzing our algorithms:
    \begin{definition}[Kurdyka-\L ojasiewicz property]\label{def:kl}
    Let $f: \real^d \to (-\infty, +\infty]$ be a proper and lower semicontinuous function.
    \begin{enumerate}
        \renewcommand{\labelenumi}{\rm{(\roman{enumi})}}
        \item $f$ is said to have the Kurdyka-\L ojasiewicz (KL) property at
        $\hat{x} \in \dom \partial f$
        if there exist $\eta \in (0, +\infty]$, a neighborhood $U$ of $\hat{x}$,
        and a function $\phi\in\Xi_{\eta}$ such that the following inequality holds:
        \begin{align}
            \label{ineq:kl}
            \phi'(f(x) - f(\hat{x})) \cdot \dist(0, \partial f(x)) \geq 1, \quad \forall x\in U \cap
            \{ x\in\real^d \mid f(\hat{x})<f(x)<f(\hat{x})+\eta\}.
        \end{align}
        \item If $f$ has the KL property at each point of $\dom\partial f$, then it is called a KL function.
    \end{enumerate}
    \end{definition}

    \begin{lemma}[Uniformized KL property~\cite{palm}]\label{lemma:uniformized-kl}
    Suppose that $f:\real^d\to(-\infty, +\infty]$ is a proper and lower semicontinuous function and let $\Gamma$ be a compact set.
    If $f$ is constant on $\Gamma$ and has the KL property at each point of $\Gamma$,
    then there exist positive scalars $\epsilon, \eta > 0$, and $\phi\in\Xi_{\eta}$ such that
    \begin{align*}
        \phi'(f(x) - f(\hat{x})) \cdot \dist(0, \partial f(x)) \geq 1,
    \end{align*}
    for any $\hat{x}\in\Gamma$ and any $x$ satisfying $\dist(x, \Gamma) < \epsilon$ and $f(\hat{x}) < f(x) < f(\hat{x}) + \eta$.
    \end{lemma}

    Next, we describe subanalytic functions.
    \begin{definition}[Subanalyticity~\cite{Bolte2007}]
        \label{def:subanalytic}
        \begin{enumerate}
            \renewcommand{\labelenumi}{\rm{(\roman{enumi})}}
            \item A subset $A$ of $\real^d$ is called semianalytic
            if each point of $\real^d$ admits a neighborhood $V$ for which $A \cap V$ assumes the following form:
            \begin{align*}
                \bigcup_{i=1}^p\bigcap_{j=1}^q\left\{ x \in V\ \middle|\ f_{ij}(x) = 0, g_{ij}(x) > 0 \right\},
            \end{align*}
            where the functions $f_{ij}, g_{ij}:V\to\real$ are real-analytic for all $1 \leq i \leq p,1 \leq j \leq q$.
            \item The set $A$ is called subanalytic if each point of $\real^d$ admits a neighborhood $V$ such that
            \begin{align*}
                A \cap V = \left\{ x\in\real^d\ \middle|\ (x, y) \in B \right\},
            \end{align*}
            where $B$ is a bounded semianalytic subset of $\real^d \times \real^m$ for some $m \geq 1$.
            \item A function $f:\real^d\to(-\infty, +\infty]$
            is called subanalytic if its graph is a subanalytic subset of $\real^d\times\real$.
        \end{enumerate}
    \end{definition}
    For instance, given a subanalytic set $S$,
    $\dist(x, S)$ is subanalytic, and every analytic function is subanalytic.
    Note that subanalytic functions are KL functions.
    See~\cite{Bierstone1988,Bolte2007} for further properties of subanalyticity.\nocite{Bierstone1988}

    \section{Proposed Methods: Bregman Proximal DC Algorithms}\label{sec:algorithms-with-bregman-distance}
    We place the following assumptions on the DC optimization problem $(\mathcal{P})$.
    Recall that $C=\interior\dom h$.

    \begin{assumption}
        \label{ass1}
        \quad
        \begin{enumerate}
            \renewcommand{\labelenumi}{\rm{(\roman{enumi})}}
            \item $h \in \mathcal{G}(C)$ with $\overline{C} = \overline{\dom h}$.
            \item $f_1:\real^d \to (-\infty, +\infty]$ is proper and convex
            with $\dom h \subset \dom (f_1 + g)$, which is $\mathcal{C}^1$ on $C$.
            \item $f_2:\real^d \to (-\infty, +\infty]$ is proper and convex.
            \item $g:\real^d \to (-\infty, +\infty]$ is proper and lower semicontinuous
            with $\dom g \cap C \neq \emptyset$.
            \item $v(\mathcal{P}) := \inf \left\{ \Psi(x) \ \middle| \ x \in \overline{C} \right\} > -\infty$.
            \item For any $\lambda > 0$, $\lambda g + h$ is supercoercieve, that is,
            \begin{align*}
                \lim_{\|u\| \to \infty} \frac{\lambda g(u) + h(u)}{\|u\|} = \infty.
            \end{align*}
        \end{enumerate}
    \end{assumption}
    Let $x\in\dom(f_1 + g)$, then $f_2(x) \leq g(x) + f_1(x) - v(\mathcal{P}) < +\infty$ due to Assumption~\ref{ass1} (v).
    Thus, $x\in\dom f_2$, \ie, $\dom (f_1 + g) \subset \dom f_2$.
    From Assumption~\ref{ass1} (ii),
    we have $C \subset \dom (f_1 + g) \subset \dom f_2$.
    Note that Assumption~\ref{ass1} (iv) holds when $\overline{C}$ is compact~\cite{bpg}.

    \subsection{Bregman Proximal DC Algorithm (BPDCA)}\label{subsec:bregman-proximal-dc-algorithm}
    To obtain the Bregman Proximal DC Algorithm (BPDCA) mapping
    for some $\lambda > 0$, we recast the objective function of $(\mathcal{P})$ via a DC decomposition:
    \begin{align*}
        \Psi(u) = f_1(u) - f_2(u) + g(u)
        =\left( \frac{1}{\lambda} h(u) + g(u) \right) - \left( \frac{1}{\lambda} h(u) - f_1(u) + f_2(u) \right),
    \end{align*}
    and, given $x\in C=\interior\dom h$ and $\xi \in \partial_{\mathrm{c}} f_2(x)$, define the mapping,
    \begin{align*}
        \mathcal{T}_{\lambda}(x) := \argmin_{u \in \overline{C}} \left\{ g(u) + \langle \nabla f_1(x) - \xi, u - x \rangle + \frac{1}{\lambda}D_h(u, x) \right\}.
    \end{align*}
    Additionally, we put the following assumption on $(\mathcal{P})$.
    \begin{assumption}
        \label{ass2}
        For all $x \in C$ and $\lambda > 0$, we have
        \begin{align*}
            \mathcal{T}_{\lambda}(x) \subset C,\quad \forall x \in C.
        \end{align*}
    \end{assumption}
    Note that Assumption~\ref{ass2} is not so restrictive because it holds when $C\equiv\real^d$.
    Under Assumptions~\ref{ass1} and~\ref{ass2}, we have the following lemma~\cite[Lemma 3.1]{bpg}.

    \begin{lemma}
        \label{lemma:mapping}
        Suppose that Assumptions~$\ref{ass1}$ and~$\ref{ass2}$ hold,
        and let $x \in C = \interior\dom h$.
        Then, the set $\mathcal{T}_{\lambda}(x)$ is a nonempty and compact subset of $C$ for any $\lambda>0$.
    \end{lemma}

    Note that when $h$ is strictly convex, $\mathcal{T}_{\lambda}(x)$ is a singleton.
    Also, when $g$ and $h$ are separable, this mapping is easily computable
    since $\mathcal{T}_{\lambda}(x)$ can be decomposed into a single-valued optimization problem,
    and often has a closed-form solution.
    For instance, when $h(x)=\frac{1}{2}\|x\|^2$, for $g(x)=\|x\|_1$, ${\mathcal T}_{\lambda}(x)$ becomes the soft-thresholding operator,
    or for $g(x)=\|x\|_0$, the hard-thresholding operator.
    Other well-known examples where this mapping has a closed-form solution are when we use an appropriate $h$ such as Burg entropy~\cite{Bauschke_2017},
    Shannon entropy~\cite{beck17}, or $h(x)=\frac{1}{4}\|x\|^4+\frac{1}{2}\|x\|^2$~\cite{bpg} for the corresponding $g$.
    Note that this $h(x)$ is not separable.
    For further examples, see~\cite[Table 2.1]{Dhillon2008}\nocite{Dhillon2008}.

    The Bregman Proximal DC Algorithm (BPDCA), which we are proposing, is listed as Algorithm~\ref{alg:bpdca}.
    \begin{algorithm}[H]
        \caption{Bregman Proximal DC Algorithm (BPDCA)}
        \label{alg:bpdca}
        \begin{algorithmic}[t]
            \normalsize
            \Input{$h \in \mathcal{G}(C)$ with $C = \interior\dom h$ such that $L$-smad for the pair $(f_1, h)$ holds on $C$.}
            \Initialization{$x^0 \in C$ and $0 < \lambda < 1 / L$.}
            \For{$k = 0, 1, 2, \ldots,$}
                \State Take any $\xi^k \in \partial_{\mathrm{c}} f_2(x^k)$ and compute
                \begin{align}
                    \label{subprob:bpdca}
                    x^{k+1} = \argmin_{x \in \overline{C}} \left\{ g(x) + \langle \nabla f_1(x^k) - \xi^k, x - x^k\rangle + \frac{1}{\lambda}D_h(x, x^k) \right\}.
                \end{align}
            \EndFor
        \end{algorithmic}
    \end{algorithm}
    As a recurrent example, $D_h(x, x^k) = \frac{1}{2}\|x - x^k\|^2$ when $h(x) = \frac{1}{2}\|x\|^2$.
    In this case, if $L$ is regarded as the Lipschitz constant for the gradient of $f_1$,
    subproblem~\eqref{subprob:bpdca} reduces to subproblem~\eqref{subprob:pdca}.
    If $f_2$ is $\mathcal{C}^1$ on $C$ and the pair $(f_1 - f_2, h)$ is $L$-smad,
    BPDCA reduces to BPG~\cite{bpg}.

    \subsection{Bregman Proximal DC Algorithm with Extrapolation (BPDCAe)}\label{subsec:bregman-proximal-dc-algorithm-with-extrapolation}
    Algorithm~\ref{alg:bpdcae}, which we are proposing, is an acceleration of BPDCA that uses the extrapolation technique~\cite{beck09,nesterov83,Nesterov2018} to solve the DC optimization problem $(\mathcal{P})$.
    \begin{algorithm}[H]
        \caption{Bregman Proximal DC Algorithm with Extrapolation (BPDCAe)}
        \label{alg:bpdcae}
        \normalsize
        \begin{algorithmic}[t]
            \Input{$h \in \mathcal{G}(C)$ with $C = \interior\dom h$ such that $L$-smad for the pair $(f_1, h)$ holds on $\real^d$.}
            \Initialization{$x^0 = x^{-1} \in \real^d, \theta_{-1}=\theta_{0}=1$, $\rho\in(0,1]$, and $0 < \lambda < 1 / L$.}
            \For{$k = 0, 1, 2, \ldots,$}
                \State Take any $\xi^k \in \partial_{\mathrm{c}} f_2(x^k)$ and compute
                \begin{align}
                    \beta_k &= \frac{\theta_{k-1} - 1}{\theta_k} \quad \mathrm{with} \quad
                    \theta_{k+1} = \frac{1 + \sqrt{1 + 4\theta_k^2}}{2},\label{update:beta}\\
                    y^k &= x^k + \beta_k(x^k - x^{k-1}).\nonumber
                \end{align}
                \If{$y^k \notin C$ or $D_h(x^k, y^k) > \rho D_h(x^{k-1}, x^k)$}
                    \State Set $\beta_k = 0$ with $\theta_{k-1} = \theta_k = 1$.
                \EndIf
                \State Compute $y^k = x^k + \beta_k(x^k - x^{k-1})$ and
                \begin{align}
                    x^{k+1} = \argmin_{y \in \overline{C}} \left\{ g(y) + \langle \nabla f_1(y^k) - \xi^k, y - y^k\rangle + \frac{1}{\lambda}D_h(y, y^k) \right\}.
                    \label{subprob:bpdcae}
                \end{align}
            \EndFor
        \end{algorithmic}
    \end{algorithm}

    When $\beta_k\equiv0$ for all $k \geq 0$, BPDCAe reduces to BPDCA.
    Here, we prefer the popular choice for the coefficients $\beta_k$ (and $\theta_k$) given in~\cite{pdcae} for acceleration.
    Accordingly, \eqref{update:beta} guarantees that $\{\beta_k\}_{k=0}^{\infty} \subset [0, 1)$ and $\sup_{k\geq0} \beta_k < 1$.
    These properties are needed to prove global subsequential convergence of the iterates (see Theorem~\ref{theorem:global-subsequential-conv-bpdcae} (ii)).
    Algorithm~\ref{alg:bpdcae} introduces a new {\it adaptive restart scheme},
    which resets $\theta_k$ and $\beta_k$ whenever
    \begin{align}
        D_h(x^k, y^k) > \rho D_h(x^{k-1}, x^k), \label{adaptive-restart}
    \end{align}
    is satisfied for a fixed $\rho\in[0,1)$.
    This adaptive restart scheme guarantees the non-increasing property for BPDCAe
    (see Lemma~\ref{lemma:obj-decrease-ex}).
    In addition, we can enforce this resetting every $K$ iterations for a given positive integer $K$.
    In our numerical experiments, we set $\{\beta_k\}_{k=0}^{\infty}$ as both the fixed and the adaptive restart schemes.

    When $C = \interior\dom h = \real^d$, $y^k$ always stays in $C$.
    However, when $C \neq \real^d$ and $x^k+\beta_k(x^k-x^{k-1}) \notin C$, Algorithm~\ref{alg:bpdcae} enforces $\beta_k=0$ and
    BPDCAe is not accelerated at the $k$th iteration.
    This operation guarantees that $y^k$ always stays in $C$.
    In practice, however, the extrapolation technique may be valid and accelerates BPDCAe.

    We define the following BPDCAe mapping for all $x, y \in C =\interior\dom h$, and $\lambda \in (0, 1/L)$:
    \begin{align*}
        \mathcal{S}_{\lambda}(x, y) := \argmin_{u \in \overline{C}} \left\{ g(u) + \langle \nabla f_1(y) - \xi, u - y \rangle + \frac{1}{\lambda}D_h(u, y) \right\},
    \end{align*}
    where $\xi \in \partial_{\mathrm{c}} f_2(x)$.
    Similarly to the case of BPDCA, we make an Assumption~\ref{ass2-ex} and can prove Lemma~\ref{lemma:mapping-ex} for $\mathcal{S}_{\lambda}(x, y) \subset \overline{C}$.

    \begin{assumption}
        \label{ass2-ex}
        For all $x, y \in C$ and $\lambda > 0$, we have
        \begin{align*}
            \mathcal{S}_{\lambda}(x, y) \subset C,\quad \forall x, y \in C.
        \end{align*}
    \end{assumption}
    \begin{lemma}
        \label{lemma:mapping-ex}
        Suppose that Assumptions~$\ref{ass1}$ and~$\ref{ass2-ex}$ hold,
        and let $x, y \in C=\interior\dom h$.
        Then, the set $\mathcal{S}_{\lambda}(x, y)$ is a nonempty and compact subset of $C$ for any $\lambda>0$.
    \end{lemma}

    \section{Convergence Analysis}\label{sec:convergence-analysis}
    Throughout this section, we will assume that the pair of functions $(f_1,h)$ is $L$-smad on $C$.

    \subsection{Properties of BPDCA}\label{subsec:properties-of-bpdca}
    First, we show the decreasing property of BPDCA mapping for $0 < \lambda L < 1$ (the argument is adapted from~\cite[Lemma 4.1]{bpg}).
    \begin{lemma}
        \label{lemma:obj-decrease}
        Suppose that Assumptions~$\ref{ass1}$ and~$\ref{ass2}$ hold.
        For any $x \in C=\interior\dom h$ and any $x^+ \in C=\interior\dom h$ defined by
        \begin{align}
            \label{eq:conv-update}
            x^+ \in \argmin_{u \in \overline{C}} \left\{ g(u) + \langle \nabla f_1(x) - \xi, u - x \rangle + \frac{1}{\lambda}D_h(u, x) \right\},
        \end{align}
        where $\xi \in \partial_{\mathrm{c}} f_2(x)$ and $\lambda > 0$, it holds that
        \begin{align}
            \label{ineq:obj-decrease}
            \lambda \Psi(x^+) \leq \lambda \Psi(x) - (1 - \lambda L) D_h(x^+, x).
        \end{align}
        In particular, the sufficiently decreasing property in the objective function value $\Psi$ is ensured when $0 < \lambda L < 1$.
    \end{lemma}

    \begin{proof}
        From the global optimality of $x^+$
        by taking $u = x \in \interior\dom h$ and $\xi \in \partial_{\mathrm{c}} f_2(x)$, we obtain
        \begin{align*}
            g(x^+) + \langle \nabla f_1(x) - \xi, x^+ - x \rangle + \frac{1}{\lambda}D_h(x^+, x) \leq g(x).
        \end{align*}
        Invoking the full Extended Descent Lemma (Lemma~\ref{lemma:descent-lemma}) for $f_1$,
        the definition of the subgradient for $f_2$, and the above inequality, we have
        \begin{align*}
            f_1(x^+) - f_2(x^+) + g(x^+) &\leq f_1(x) - f_2(x) + \langle \nabla f_1(x) - \xi, x^+ - x \rangle + LD_h(x^+, x) + g(x^+)\\
            &\leq f_1(x) - f_2(x) + LD_h(x^+, x) + g(x) - \frac{1}{\lambda}D_h(x^+, x)\\
            &= f_1(x) - f_2(x) + g(x) - \left( \frac{1}{\lambda} - L \right)D_h(x^+, x),
        \end{align*}
        for $\Psi = f_1 - f_2 + g$.
        The last statement follows with $0 < \lambda L < 1$.
        \qed
    \end{proof}

    Proposition~\ref{prop:bpdca-property} follows immediately from Lemma~\ref{lemma:obj-decrease}, as in~\cite{bpg}.

    \begin{proposition}\label{prop:bpdca-property}
    Suppose that Assumptions~$\ref{ass1}$ and~$\ref{ass2}$ hold.
    Let $\{x^k\}_{k=0}^{\infty}$ be a sequence generated by BPDCA with $0 < \lambda L < 1$.
    Then, the following statements hold:
    \begin{enumerate}
        \renewcommand{\labelenumi}{\rm{(\roman{enumi})}}
        \item The sequence $\{\Psi(x^k)\}_{k=0}^{\infty}$ is non-increasing.
        \item $\sum_{k=1}^{\infty}D_h(x^k, x^{k-1}) < \infty$; hence, the sequence $\{D_h(x^k, x^{k-1})\}_{k=0}^{\infty}$ converges to zero.
        \item $\min_{1\leq k \leq n}D_h(x^k, x^{k-1}) \leq \frac{\lambda}{n}\left( \frac{\Psi(x^0) - \Psi_*}{1 - \lambda L} \right)$,
        where $\Psi_* = v(\mathcal{P}) > -\infty$ \rm{(by Assumption~\ref{ass1} (v))}.
    \end{enumerate}
    \end{proposition}

    \subsection{Convergence Analysis of BPDCA}\label{subsec:convergence-analysis-of-bpdca}
    Suppose that the following conditions hold.
    \begin{assumption}
        \label{ass4}
        \quad
        \begin{enumerate}
            \renewcommand{\labelenumi}{\rm{(\roman{enumi})}}
            \item $\dom h = \real^d$ and $h$ is $\sigma$-strongly convex on $\real^d$.
            \item $\nabla h$ and $\nabla f_1$ are Lipschitz continuous on any bounded subset of $\real^d$.
            \item The objective function $\Psi$ is level-bounded; \ie, for any $r \in\real$, the lower level sets $\{ x\in\real^d\ |\ \Psi(x)\leq r \}$ are bounded.
        \end{enumerate}
    \end{assumption}

    Since $C = \interior\dom h = \real^d$ under Assumption~\ref{ass4} (i),
    Assumptions~\ref{ass2} and~\ref{ass2-ex} are automatically fulfilled.
    For nonconvex functions, we use the limiting subdifferential~\cite{varAna}(Definition~\ref{def:limiting-subdiff}).
    Inspired by Fermat's rule~\cite[Theorem 10.1]{varAna},
    we define the limiting critical points and the limiting stationary points of $\Psi$.
    \begin{definition}
        \label{def:critical-stationary}
        We say that $\tilde{x}$ is a limiting critical point of $(\mathcal{P})$ with $C \equiv \real^d$ if
        \begin{equation}
            \label{def:critical}
            0 \in \nabla f_1(\tilde{x}) - \partial_{\mathrm{c}} f_2(\tilde{x}) + \partial g(\tilde{x}).
        \end{equation}
        The set of all limiting critical points of $(\mathcal{P})$ is denoted by $\mathcal{X}$.
        In addition, we say that $\tilde{x}$ is a limiting stationary point of $(\mathcal{P})$ with $C \equiv \real^d$ if
        \begin{equation}
            \label{def:stationary}
            0 \in \partial \Psi(\tilde{x}).
        \end{equation}
    \end{definition}

    Although the limiting stationary points are sometimes called the limiting critical points in some papers, for example~\cite[Definition 1 (iv)]{palm}, we distinguish these two terms.
    The reasons are the following:
    When $\Psi$ is convex, we call $\tilde{x}$ a stationary point if it satisfies $0\in\partial_{\mathrm{c}}\Psi(\tilde{x})$.
    Because \eqref{def:stationary} is its natural extension by replacing $\partial_{\mathrm{c}}\Psi$ with $\partial\Psi$,
    we use the terminology ``limiting stationary point'' after~\cite[Definition 6.1.4]{cui2021modern}.
    We similarly name $\tilde{x}$ satisfying \eqref{def:critical}:
    When $g$ is convex, we call $\tilde{x}$ a critical point if it satisfies $0\in\nabla f_1(\tilde{x}) - \partial_{\mathrm{c}} f_2(\tilde{x}) + \partial_{\mathrm{c}} g(\tilde{x})$.
    Because \eqref{def:critical} is its natural extension by replacing $\partial_{\mathrm{c}}g$ with $\partial g$,
    we use the terminology ``limiting critical point.''

    The limiting stationary point is a first-order necessary condition for the local optimality.
    This relation is known as the generalized Fermat’s rule~\cite[Theorem 10.1]{varAna}.
    We can deduce $\partial(g - f_2)(x) \subseteq \partial g(x) - \partial_{\mathrm{c}} f_2(x)$ from~\cite[Corollary 3.4]{Mordukhovich06}.
    Plugging it into~\cite[Corollary 10.9]{varAna}, it generally holds that $\partial \Psi(x) \subseteq \nabla f_1(x) - \partial_{\mathrm{c}} f_2(x) + \partial g(x)$ for all $x\in\real^d$.
    It implies the limiting critical point is weaker than the limiting stationary point.
    When $f_2$ is $\mathcal{C}^1$ on $\real^d$, it holds that $\partial\Psi(x) \equiv \nabla f_1(x) - \nabla f_2(x) + \partial g(x)$
    from~\cite[Corollary 10.9]{varAna} or~\cite[Proposition 1.107 (ii)]{mordukhovich2006variational} and the definition of the limiting subdifferentials of $f_2$ and $g$.
    Thus, every limiting critical point is a limiting stationary point when $f_2$ is $\mathcal{C}^1$.

    Next, using Lemma~\ref{lemma:obj-decrease} and Proposition~\ref{prop:bpdca-property},
    we will show global subsequential convergence of the iterates to a limiting critical point of the problem $(\mathcal{P})$.
    We can easily see that Theorem~\ref{theorem:global-subsequential-conv-bpdca} (i) holds from the level-boundedness of $\Psi$.
    Theorem~\ref{theorem:global-subsequential-conv-bpdca} (iii) and (vi) will play an essential role
    in determining the global convergence and the rate of convergence of BPDCA.

    \begin{theorem}[Global subsequential convergence of BPDCA]
        \label{theorem:global-subsequential-conv-bpdca}
        Suppose that Assumptions~$\ref{ass1}$,~$\ref{ass2}$, and $\ref{ass4}$ hold.
        Let $\{x^k\}_{k=0}^{\infty}$ be a sequence generated by BPDCA with $0 < \lambda L < 1$ for solving $(\mathcal{P})$.
        Then, the following statements hold:
        \begin{enumerate}
            \renewcommand{\labelenumi}{\rm{(\roman{enumi})}}
            \item The sequence $\{x^k\}_{k=0}^{\infty}$ is bounded.
            \item The sequence $\{\xi^k\}_{k=0}^{\infty}$ is bounded.
            \item $\lim_{k\to\infty}\|x^{k+1} - x^k\| = 0$.
            \item Any accumulation point of $\{x^k\}_{k=0}^{\infty}$ is a limiting critical point of $(\mathcal{P})$.
        \end{enumerate}
    \end{theorem}

    \begin{proof}
    (i)
        From Proposition~\ref{prop:bpdca-property}, we obtain
        $\Psi(x^k) \leq \Psi(x^0)$ for all $k \in \mathbb{N}$,
        which shows that $\{x^k\}_{k=0}^{\infty}$ is bounded
        from Assumption~\ref{ass4} (iii).

        (ii)
        From Assumption~\ref{ass1} (ii),~\ref{ass4} (i), and the convexity of $f_2$, $\dom f_2=\real^d$ and $\partial_{\mathrm{c}} f_2 (x^k) \neq \emptyset$.
        Suppose, for the sake of proof by contradiction, that $\{\xi^{k}\}_{k=0}^{\infty}$ is unbounded, \ie, $\|\xi^{k}\|\to\infty$ as $k\to\infty$.
        By the definition of the subgradients of convex functions, we have that, for any $y\in\real^d$,
        \begin{align}
            \label{ineq:subgradient}
            f_2(y) \geq f_2(x^{k}) + \langle \xi^{k}, y - x^{k} \rangle.
        \end{align}
        Assume for a moment that $\|\xi^k\| \neq 0$.
        Letting $\{d^{k}\}_{k=0}^{\infty}$ be the subsequence given by $d^{k} = \xi^{k}/\|\xi^{k}\|$ and
        substituting $x^{k} + d^{k} = x^{k} + \xi^{k}/\|\xi^{k}\|$ into $y$ in~\eqref{ineq:subgradient}, we obtain
        \begin{align*}
            f_2(x^{k} + d^{k}) \geq f_2(x^{k}) + \left\langle \xi^{k}, d^k \right\rangle
            = f_2(x^{k}) + \|\xi^{k}\|,
        \end{align*}
        which is also true when $\|\xi^k\| = 0$ by defining $d^k = 0$.
        By taking $k\to\infty$, we obtain
        \begin{align}
            \label{ineq:subgrad-limsup}
            \limsup_{k\to\infty}\|\xi^{k}\| \leq\limsup_{k\to\infty}\left( f_2(x^{k} + d^{k}) - f_2(x^{k}) \right).
        \end{align}
        We can take a compact set $S$ such that $\{x^{k} + d^{k}\}_{k=0}^{\infty} \subset S$, since $\{x^{k} + d^{k}\}_{k=0}^{\infty}$ is bounded.
        For $\bar{x} \in \argmax_{x \in S} f_2(x)$,
        since $f_2$ is continuous because of its convexity on $\real^d$ and $\{x^k\}_{k=0}^{\infty}$ is bounded, it holds that
        \begin{align}
            \label{ineq:bounded-subgrad-limsup}
            \limsup_{k\to\infty}\left( f_2(x^{k} + d^{k}) - f_2(x^{k}) \right)
            \leq f_2(\bar{x}) - \bar{f}_2 < \infty,
        \end{align}
        for some value $\bar{f}_2 \leq f_2(x^k), k\geq 0$.
        \eqref{ineq:subgrad-limsup} and \eqref{ineq:bounded-subgrad-limsup} contradict to $\|\xi^{k}\|\to\infty$.

        (iii)
        From~\eqref{ineq:obj-decrease}, we obtain
        \begin{align}
            \Psi(x^{k-1}) - \Psi(x^k)
            &\geq \left( \frac{1}{\lambda} - L \right)D_h(x^k, x^{k-1})\nonumber\\
            &\geq \left( \frac{1}{\lambda} - L \right)\frac{\sigma}{2}\|x^k - x^{k-1}\|^2,\label{ineq:sufficient-decrease}
        \end{align}
        where the last inequality holds since $h$ is a $\sigma$-strongly convex function from Assumption~\ref{ass4} (i).
        Summing the above inequality from $k=1$ to $\infty$, we obtain
        \begin{align*}
            \left( \frac{1}{\lambda} - L \right)\sum_{k=1}^{\infty}\frac{\sigma}{2}\|x^k - x^{k-1}\|^2
            \leq \Psi(x^0) - \liminf_{n\to\infty}\Psi(x^n)
            \leq \Psi(x^0) - v(\mathcal{P}) < \infty,
        \end{align*}
        which shows that $\lim_{k\to\infty}\|x^{k+1} - x^k\| = 0$.

        (iv)
        Let $\tilde{x}$ be an accumulation point of $\{x^k\}_{k=0}^{\infty}$
        and let $\{x^{k_j}\}$ be a subsequence such that $\lim_{j\to\infty}x^{k_j} = \tilde{x}$.
        Then, from the first-order optimality condition of subproblem~\eqref{subprob:bpdca}
        under Assumption~\ref{ass2}, we have
        \begin{align*}
            0 \in \partial g(x^{k_j+1}) + \nabla f_1(x^{k_j}) - \xi^{k_j} + \frac{1}{\lambda}\left(\nabla h(x^{k_j + 1}) - \nabla h(x^{k_j}) \right).
        \end{align*}
        Therefore,
        \begin{align}
            \label{con:1st-k_j}
            \xi^{k_j} + \frac{1}{\lambda}\left(\nabla h(x^{k_j}) - \nabla h(x^{k_j+1}) \right) \in \partial g(x^{k_j+1}) + \nabla f_1(x^{k_j}).
        \end{align}
        From the boundedness of $\{x^{k_j}\}$ and the Lipschitz continuity of $\nabla h$ on a bounded subset of $\real^d$,
        there exists $A_0 > 0$ such that
        \begin{align*}
            \left\| \frac{1}{\lambda}\left(\nabla h(x^{k_j}) - \nabla h(x^{k_j+1}) \right) \right\|
            &\leq \frac{A_0}{\lambda}\| x^{k_j+1} - x^{k_j} \|.
        \end{align*}
        Therefore,
        using $\|x^{k_j+1} - x^{k_j}\| \to 0$,
        we obtain
        \begin{align}
            \label{lim:1st-con}
            \frac{1}{\lambda}\left(\nabla h(x^{k_j}) - \nabla h(x^{k_j+1}) \right) \to 0.
        \end{align}
        Note that the sequence $\{\xi^{k_j}\}$ is bounded due to (ii).
        Thus, by taking the limit as $j\to\infty$ or more precisely, its subsequence, we can assume without loss of generality
        that $\lim_{j\to\infty}\xi^{k_j} =:\tilde{\xi}$ exists, which belongs to $\partial_{\mathrm{c}} f_2(\tilde{x})$
        since $f_2$ becomes continuous due to its convexity on $\real^d$.
        Using this and~\eqref{lim:1st-con},
        we can take the limit of~\eqref{con:1st-k_j}.
        Setting $\|x^{k_j+1} - x^{k_j}\| \to 0$ and invoking
        the lower semicontinuity of $g$ and $\nabla f_1$,
        we obtain $\tilde{\xi}\in\partial g(\tilde{x})+\nabla f_1(\tilde{x})$.
        Therefore, $0 \in \partial g(\tilde{x}) + \nabla f_1(\tilde{x}) - \partial_{\mathrm{c}} f_2(\tilde{x})$,
        which shows that $\tilde{x}$ is a limiting critical point of $(\mathcal{P})$.
        \qed
    \end{proof}

    We can estimate the objective value at an accumulation point from
    $\liminf_{j\to\infty}\Psi(x^{k_j})$ and $\limsup_{j\to\infty}\Psi(x^{k_j})$.
    Consequently, we can prove that $\Psi$ is constant on the set of accumulation points of BPDCA.

    \begin{proposition}
        \label{prop:zeta}
        Suppose that Assumptions~$\ref{ass1}$,~$\ref{ass2}$, and $\ref{ass4}$ hold.
        Let $\{x^k\}_{k=0}^{\infty}$ be a sequence generated by BPDCA with $0 < \lambda L < 1$ for solving $(\mathcal{P})$.
        Then, the following statements hold:
        \begin{enumerate}
            \renewcommand{\labelenumi}{\rm{(\roman{enumi})}}
            \item $\zeta := \lim_{k\to\infty}\Psi(x^k)$ exists.
            \item $\Psi \equiv \zeta$ on $\Omega$, where $\Omega$ is the set of accumulation points of $\{x^k\}_{k=0}^{\infty}$.
        \end{enumerate}
    \end{proposition}

    \begin{proof}
    (i)
        From Assumption~\ref{ass1} (v) and Proposition~\ref{prop:bpdca-property} (i),
        the sequence $\{\Psi(x^k)\}_{k=0}^{\infty}$ is bounded from below and non-increasing.
        Consequently, $\zeta := \lim_{k\to\infty}\Psi(x^k)$ exists.

        (ii)
        Take any $\hat{x}\in\Omega$, that is $\lim_{j\to\infty}x^{k_j}=\hat{x}$.
        From~\eqref{subprob:bpdca}, it follows that
        \begin{align*}
            g(x^{k})&
            + \langle \nabla f_1(x^{k-1}) - \xi^{k-1}, x^{k} - x^{k-1}\rangle
            + \frac{1}{\lambda}D_h(x^{k}, x^{k-1})\\
            &\leq g(\hat{x}) + \langle \nabla f_1(x^{k-1}) - \xi^{k-1}, \hat{x} - x^{k-1}\rangle
            + \frac{1}{\lambda}D_h(\hat{x}, x^{k-1}).
        \end{align*}
        From the above inequality and the fact that $f_1$ is convex at $x^k$, we obtain
        \begin{align*}
            g(x^{k}) + f_1(x^k) \leq g(\hat{x})
            &+ \langle \nabla f_1(x^{k-1}) - \xi^{k-1}, \hat{x} - x^{k}\rangle
            + \frac{1}{\lambda}D_h(\hat{x}, x^{k-1}) - \frac{1}{\lambda}D_h(x^{k}, x^{k-1})\\
            &+ f_1(\hat{x})+\langle\nabla f_1(x^k),x^k-\hat{x}\rangle.
        \end{align*}
        Substituting $k_j$ for $k$ and limiting $j$ to $\infty$, we have, from Proposition~\ref{prop:bpdca-property} (ii),
        \begin{align*}
            \limsup_{j\to\infty} \left( g(x^{k_j}) + f_1(x^{k_j}) \right)\leq g(\hat{x}) + f_1(\hat{x}),
        \end{align*}
        which provides $\limsup_{j\to\infty}\Psi(x^{k_j})\leq\Psi(\hat{x})$ from the continuity of $- f_2$ since $f_2$ is convex.
        Combining this and the lower semicontinuity of $\Psi$ yields $\Psi(x^{k_j})\to\Psi(\hat{x})=:\zeta$ as $j\to\infty$.
        Since $\hat{x}\in\Omega$ is arbitrary, we conclude that $\Psi \equiv \zeta$ on $\Omega$.
        \qed
    \end{proof}

    To discuss the global convergence of BPDCA, we will suppose either of the following two assumptions.
    \begin{assumption}
        \label{ass5}
        $f_2$ is continuously differentiable on an open set $\mathcal{N}_0\subset\real^d$ that contains
        the set of all limiting critical points of $\Psi$, \ie, $\mathcal{X}$.
        Furthermore, $\nabla f_2$ is locally Lipschitz continuous on $\mathcal{N}_0$.
    \end{assumption}
    \begin{assumption}
        \label{ass5g}
        $g$ is differentiable on $\real^d$ and
        $\nabla g$ is locally Lipschitz continuous on an open set $\mathcal{N}_0\subset \real^d$
        that contains the set of all limiting stationary points of $-\Psi$.
    \end{assumption}

    Assumption~\ref{ass5} is nonrestrictive because many functions in~\cite{pdcae}, including the $f_2$
    in our numerical experiments, satisfy it.
    Thus, let us discuss the global convergence of Algorithm~\ref{alg:bpdca} under Assumption~\ref{ass5} by following the argument presented in ~\cite{pdcae}.
    Note that every limiting critical point is a limiting stationary point from the differentiability of $f_2$ under Assumption~\ref{ass5}.
    \begin{theorem}[Global convergence of BPDCA under the local differentiability of $f_2$]
    \label{theorem:global-convergence-bpdca}
    Suppose that Assumptions~$\ref{ass1}$,~$\ref{ass2}$,~$\ref{ass4}$, and~$\ref{ass5}$ hold and that $\Psi$ is a KL function.
    Let $\{x^k\}_{k=0}^{\infty}$ be a sequence generated by BPDCA with $0 < \lambda L < 1$ for solving $(\mathcal{P})$.
    Then, the following statements hold:
    \begin{enumerate}
        \renewcommand{\labelenumi}{\rm{(\roman{enumi})}}
        \item $\lim_{k\to\infty}\dist(0, \partial \Psi(x^k)) = 0$.
        \item The sequence $\{x^k\}_{k=0}^{\infty}$ converges to a limiting stationary point of $(\mathcal{P})$;
        moreover, $\sum_{k=1}^{\infty} \|x^k - x^{k-1}\| < \infty$.
    \end{enumerate}
    \end{theorem}

    \begin{proof}
    (i)
        Since $\{x^k\}_{k=0}^{\infty}$ is bounded
        and $\Omega$ is the set of accumulation points of $\{x^k\}_{k=0}^{\infty}$,
        we have
        \begin{align}
            \label{eq:convergence-dist-x-omega}
            \lim_{k\to\infty}\dist(x^k, \Omega) = 0.
        \end{align}
        From Theorem~\ref{theorem:global-subsequential-conv-bpdca} (iv), we also have $\Omega\subseteq\mathcal{X}$.
        Thus, for any $\mu>0$, there exists $k_0>0$ such that $\dist(x^k, \Omega)<\mu$ and $x^k\in\mathcal{N}_0$
        for any $k \geq k_0$,
        where $\mathcal{N}_0$ is defined in Assumption~\ref{ass5}.
        As for $\mathcal{N}_0$, since $\Omega$ is compact from the boundedness of $\{x^k\}_{k=0}^{\infty}$,
        by decreasing $\mu$, if needed, we can suppose without loss of generality that $\nabla f_2$ is globally Lipschitz continuous on $\mathcal{N}:=\{ x\in\mathcal{N}_0 \mid \dist(x, \Omega)<\mu\}$.

        The subdifferential of $\Psi$ at $x^k$ for $k \geq k_0$ is
        \begin{align}
            \label{partial-psi}
            \partial \Psi(x^k) = \nabla f_1(x^k) - \nabla f_2(x^k) + \partial g(x^k).
        \end{align}
        Moreover, considering the first-order optimality condition of subproblem~\eqref{subprob:bpdca}, we have that, for any $k \geq k_0 + 1$,
        \begin{align*}
            \frac{1}{\lambda}\left(\nabla h(x^{k-1}) - \nabla h(x^k) \right) - \nabla f_1(x^{k-1}) + \nabla f_2(x^{k-1}) \in \partial g(x^k),
        \end{align*}
        since $f_2$ is $\mathcal{C}^1$ on $\mathcal{N}$ and $x^{k-1}\in\mathcal{N}$ for any $k \geq k_0 + 1$.
        Using the above and~\eqref{partial-psi}, we see that
        \begin{align*}
            \frac{1}{\lambda}\left(\nabla h(x^{k-1}) - \nabla h(x^k) \right) + \nabla f_1(x^k) - \nabla f_1(x^{k-1}) + \nabla f_2(x^{k-1}) - \nabla f_2(x^k) \in \partial \Psi(x^k).
        \end{align*}
        From the global Lipschitz continuity of $\nabla f_1, \nabla f_2$, and $\nabla h$,
        there exists $A_1 > 0$ such that
        \begin{align}
            \label{ineq:dist-euc}
            \dist(0, \partial\Psi(x^k)) \leq A_1\|x^k - x^{k-1}\|,
        \end{align}
        where $k \geq k_0 + 1$.
        From Theorem~\ref{theorem:global-subsequential-conv-bpdca} (iii), we conclude that $\lim_{k\to\infty}\dist(0, \partial \Psi(x^k)) = 0$.

        (ii)
        From Theorem~\ref{theorem:global-subsequential-conv-bpdca} (iv), it is sufficient to prove that $\{x^k\}_{k=0}^{\infty}$ is convergent.
        Here, consider the case in which there exists a positive integer $k > 0$ such that $\Psi(x^k) = \zeta$.
        From Proposition~\ref{prop:bpdca-property} (i) and Proposition~\ref{prop:zeta} (i),
        the sequence $\{\Psi(x^k)\}_{k=0}^{\infty}$ is non-increasing and converges to $\zeta$.
        Hence, for any $\hat{k}\geq0$, we have $\Psi(x^{k + \hat{k}}) = \zeta$.
        By recalling~\eqref{ineq:sufficient-decrease}, we conclude that there exists a positive scalar $A_2$ such that
        \begin{align}
            \label{ineq:sufficient-decrease-rho}
            \Psi(x^{k-1}) - \Psi(x^k)
            \geq A_2\|x^k - x^{k-1}\|^2,\quad \forall k \in \mathbb{N}.
        \end{align}
        From~\eqref{ineq:sufficient-decrease-rho}, we obtain $x^k = x^{k+\hat{k}}$ for any $\hat{k}\geq0$,
        meaning that $\{x^k\}_{k=0}^{\infty}$ is finitely convergent.

        Next, consider the case in which $\Psi(x^k) > \zeta$ for all $k\geq0$.
        Since $\{x^k\}_{k=0}^{\infty}$ is bounded,
        $\Omega$ is a compact subset of $\dom\partial\Psi$ and $\Psi \equiv \zeta$ on $\Omega$ from Proposition~\ref{prop:zeta} (ii).
        From Lemma~\ref{lemma:uniformized-kl}
        and since $\Psi$ is a KL function, there exist a positive scalar $\epsilon > 0$ and
        a continuous concave function $\phi \in \Xi_{\eta}$ with $\eta > 0$ such that
        \begin{align}
            \label{ineq:phi-kl}
            \phi'(\Psi(x) - \zeta) \cdot \dist(0, \partial \Psi(x)) \geq 1,
        \end{align}
        for all $x \in U$, where
        \begin{align*}
            U = \left\{ x\in\real^d\ \middle|\ \dist(x, \Omega) < \epsilon \right\} \cap
            \left\{ x\in\real^d\ \middle|\ \zeta < \Psi(x) < \zeta + \eta \right\}.
        \end{align*}
        From~\eqref{eq:convergence-dist-x-omega}, there exists $k_1 > 0$ such that $\dist(x^k, \Omega) < \epsilon$ for any $k \geq k_1$.
        Since $\{\Psi(x^k)\}_{k=0}^{\infty}$ is non-increasing and converges to $\zeta$,
        there exists $k_2 > 0$ such that $\zeta < \Psi(x^k) < \zeta + \eta$ for all $ k \geq k_2$.
        Taking $\bar{k} = \max\{k_0+1, k_1, k_2\}$, then $\{x^k\}_{k \geq \bar{k}}$ belongs to $U$.
        Hence, from~\eqref{ineq:phi-kl}, we obtain
        \begin{align}
            \label{ineq:phi-kl-iter}
            \phi'(\Psi(x^k) - \zeta) \cdot \dist(0, \partial \Psi(x^k)) \geq 1,\quad \forall k \geq \bar{k}.
        \end{align}
        Since $\phi$ is a concave function, we see that for any $k \geq\bar{k}$,
        \begin{align*}
            \left[ \phi(\Psi(x^k) - \zeta) - \phi(\Psi(x^{k+1}) - \zeta) \right] &\cdot \dist(0, \partial \Psi(x^k))\\
            &\geq \phi'(\Psi(x^k) - \zeta) \cdot \dist(0, \partial \Psi(x^k)) \cdot \left( \Psi(x^k) - \Psi(x^{k+1}) \right)\\
            &\geq \Psi(x^k) - \Psi(x^{k+1})\\
            &\geq A_2\|x^{k+1} - x^k\|^2,
        \end{align*}
        where the second inequality holds from~\eqref{ineq:phi-kl-iter} and the fact that $\{\Psi(x^k)\}_{k=0}^{\infty}$ is non-increasing,
        and the last inequality holds from~\eqref{ineq:sufficient-decrease-rho}.
        From the above inequality and~\eqref{ineq:dist-euc}, we obtain
        \begin{align}
            \label{ineq:constant-estimate}
            \|x^{k+1} - x^k\|^2 \leq \frac{A_1}{A_2} \left( \phi(\Psi(x^k) - \zeta) - \phi(\Psi(x^{k+1}) - \zeta) \right) \|x^k - x^{k-1}\|.
        \end{align}
        Taking the square root of~\eqref{ineq:constant-estimate} and using the inequality of the arithmetic and geometric means, we find that
        \begin{align*}
            \|x^{k+1} - x^k \| &\leq \sqrt{\frac{A_1}{A_2} \left( \phi(\Psi(x^k) - \zeta) - \phi(\Psi(x^{k+1}) - \zeta) \right)} \cdot \sqrt{\|x^k - x^{k-1}\|}\\
            &\leq \frac{A_1}{2A_2} \left( \phi(\Psi(x^k) - \zeta) - \phi(\Psi(x^{k+1}) - \zeta) \right) + \frac{1}{2}\|x^k - x^{k-1}\|.
        \end{align*}
        This shows that
        \begin{align}
            \label{ineq:summing}
            \frac{1}{2}\|x^{k+1} - x^k \|\leq\frac{A_1}{2A_2} \left( \phi(\Psi(x^k) - \zeta) - \phi(\Psi(x^{k+1}) - \zeta) \right)
            + \frac{1}{2}\|x^k - x^{k-1}\| - \frac{1}{2}\|x^{k+1} - x^k \|.
        \end{align}
        Summing~\eqref{ineq:summing} from $k=\bar{k}$ to $\infty$, we have
        \begin{align*}
            \sum_{k=\bar{k}}^{\infty}\|x^{k+1} - x^k \|
            \leq \frac{A_1}{A_2} \phi(\Psi(x^{\bar{k}}) - \zeta) + \|x^{\bar{k}} - x^{\bar{k}-1}\| < \infty,
        \end{align*}
        which implies that $\sum_{k=1}^{\infty} \|x^k - x^{k-1}\| < \infty$, \ie,
        the sequence $\{x^k\}_{k=0}^{\infty}$ is a Cauchy sequence.
        Thus, $\{x^k\}_{k=0}^{\infty}$ converges to a limiting critical point of $(\mathcal{P})$ from Theorem~\ref{theorem:global-subsequential-conv-bpdca} (iv).
        Because every limiting critical point is a limiting stationary point from the differentiability of $f_2$, $\{x^k\}_{k=0}^{\infty}$ converges to a limiting stationary point of $(\mathcal{P})$.
        \qed
    \end{proof}

    Next, suppose that Assumption~\ref{ass5g} holds instead of Assumption~\ref{ass5}.
    Here, we can show the global convergence of BPDCA by referring to~\cite[Theorem 3.4]{hoai2018}.
    We will use subanalyticity instead of the KL property in the proof.
    \begin{theorem}[Global convergence of BPDCA under the local differentiability of $g$]
    Suppose that Assumptions~$\ref{ass1}$,~$\ref{ass2}$,~$\ref{ass4}$, and~$\ref{ass5g}$ hold
    and that $\Psi$ is subanalytic.
    Let $\{x^k\}_{k=0}^{\infty}$ be a sequence generated by BPDCA with $0 < \lambda L < 1$ for solving $(\mathcal{P})$.
    Then, the sequence $\{x^k\}_{k=0}^{\infty}$ converges to a limiting critical point of $(\mathcal{P})$;
    moreover, $\sum_{k=1}^{\infty} \|x^k - x^{k-1}\| < \infty$.
    \label{theorem:global-convergence-bpdca-g}
    \end{theorem}
    \begin{proof}
        Since $g$ is differentiable, $g$ is continuous on $\real^d$.
        Since the convexity of $f_1$ and $f_2$ implies their continuity, $\Psi$ is continuous on $\real^d$.

        Let $\{\xi^k\}_{k=0}^{\infty}$ on $\real^d$ be a sequence of subgradients of $f_2$.
        From Theorem~\ref{theorem:global-subsequential-conv-bpdca} (i) and (ii), $\{x^k\}_{k=0}^{\infty}$ and $\{\xi^k\}_{k=0}^{\infty}$ are bounded.
        Let $\tilde{x}$ be a limiting stationary point of $-\Psi$ and
        $B(\tilde{x}, \epsilon_0)$ be an open ball with center $\tilde{x}$ and radius $\epsilon_0 > 0$.
        Since $\nabla g$ is locally Lipschitz continuous, and Assumption~\ref{ass4} (ii) holds, for $\lambda > 0$, there exist $\kappa_0 > 0$ and $\epsilon_0 > 0$ such that
        \begin{align}
            \label{ineq:g-h}
            \left\|\nabla \left(g + \frac{1}{\lambda}h\right)(u) - \nabla \left(g + \frac{1}{\lambda}h\right)(v)\right\|\leq\kappa_0\|u-v\|,
            \quad \forall u, v \in B(\tilde{x}, \epsilon_0).
        \end{align}
        From Assumption~\ref{ass1} (v), $-\Psi$ is finite.
        In addition,
        by recalling the continuity and subanalyticity of $-\Psi$ on $B(\tilde{x}, \epsilon_0)$,
        we can apply~\cite[Theorem 3.1]{Bolte2007}\nocite{Bolte2007} to the subanalytic function $-\Psi$ and
        obtain $\nu_0 > 0$ and $\theta_0\in[0, 1)$ such that
        \begin{align}
            \label{ineq:psi-nu}
            |\Psi(u) - \zeta|^{\theta_0} \leq \nu_0\|\hat{x}\|,\quad \forall u\in B(\tilde{x}, \epsilon_0),\quad\hat{x}\in\partial(-\Psi)(u),
        \end{align}
        where $\zeta = \Psi(\tilde{x})$.

        Let $\Omega$ be the set of accumulation points of $\{x^k\}_{k=0}^{\infty}$.
        Since $\Omega$ is compact,
        $\Omega$ can be covered by a finite number of $B(\tilde{x}_j, \epsilon_j)$ with $\tilde{x}_j\in\Omega$ and $\epsilon_j>0$, $j=1,\ldots,p$.
        From Theorem~\ref{theorem:global-subsequential-conv-bpdca} (iv), $\tilde{x}_j\in\Omega, j=1,\ldots,p$ are limiting critical points of $(\mathcal{P})$.
        Hence,~\eqref{ineq:g-h} with $\kappa_j>0$ and $\epsilon_j>0$ and~\eqref{ineq:psi-nu} with $\nu_j > 0$ and $\theta_j\in[0, 1)$ hold for $j=1,\ldots,p$.
        Letting $\epsilon > 0$ be a sufficiently small constant,
        we obtain
        \begin{align*}
            \left\{ x\in\real^d\ \middle|\ \dist(x, \Omega) < \epsilon \right\}\subset\bigcup_{j=1}^{p}B(\tilde{x}_j, \epsilon_j).
        \end{align*}
        From~\eqref{eq:convergence-dist-x-omega}, there exists $k_1 > 0$ such that $\dist(x^k, \Omega) < \epsilon$ for any $k \geq k_1$;
        hence, $x^k\in\bigcup_{j=1}^{p}B(\tilde{x}_j, \epsilon_j)$ for any $k \geq k_1$.
        From Theorem~\ref{theorem:global-subsequential-conv-bpdca} (iii),
        letting $\bar{\epsilon} > 0$ be a sufficiently small constant,
        there exists $k_2 > 0$ such that $\|x^k - x^{k+1}\|\leq\frac{\bar{\epsilon}}{2}$ for any $k \geq k_2$.
        Therefore, redefining $\bar{\epsilon}, \epsilon_j, j=1,\ldots,p$ and
        relabeling if necessary, we can assume without loss of generality that
        \begin{align*}
            x^k \in\bigcup_{j=1}^{p}B(\tilde{x}_j, \frac{\epsilon_j}{2}) \quad \mathrm{and} \quad \|x^k - x^{k+1}\|\leq\frac{\bar{\epsilon}}{2},\quad\forall k \geq \bar{k},
        \end{align*}
        where $\bar{\epsilon} = \min_{j=1,\ldots,p}\epsilon_j$ and $\bar{k} = \max\{k_1, k_2\}$,
        which implies $x^k\in B(\tilde{x}_{j_k}, \epsilon_{j_k}/2), j_k \in\{1, \ldots, p\}$ and hence $x^{k+1}\in B(\tilde{x}_{j_k}, \epsilon_{j_k})$.
        Thus, from~\eqref{ineq:g-h} and~\eqref{ineq:psi-nu}, we have
        \begin{align}
            &\left\|\nabla \left(g + \frac{1}{\lambda}h\right)(x^k) - \nabla \left(g + \frac{1}{\lambda}h\right)(x^{k+1})\right\|\leq\kappa\|x^k-x^{k+1}\|,\label{ineq:g-h-xk} \\
            &|\Psi(x^k) - \zeta|^{\theta} \leq \nu \|\hat{x}^k\|,\quad \hat{x}^k\in\partial(-\Psi)(x^k), \quad \forall k \geq \bar{k},\label{ineq:psi-xk}
        \end{align}
        where $\kappa = \max_{j=1,\ldots,p}\kappa_j, \nu = \max_{j=1,\ldots,p}\nu_j$, and $\theta = \max_{j=1,\ldots,p}\theta_j$.
        From~\eqref{subprob:bpdca}, we find that
        \begin{align*}
            0 = \nabla g(x^{k+1}) + \nabla f_1(x^{k}) - \xi^{k} + \frac{1}{\lambda}\left(\nabla h(x^{k + 1}) - \nabla h(x^{k}) \right),
        \end{align*}
        which implies
        \begin{align*}
            \nabla g(x^{k+1}) - \nabla g(x^k) + \frac{1}{\lambda}\left(\nabla h(x^{k + 1}) - \nabla h(x^{k}) \right)
            = \xi^k - \nabla f_1(x^{k}) - \nabla g(x^k)
            \in \partial(-\Psi)(x^k),
        \end{align*}
        where we have used $\partial(-\Psi)(x^k) = \partial_{\mathrm{c}} f_2(x^k) - \nabla f_1(x^k) - \nabla g(x^k)$, which comes from the convexity of $f_2$.
        Using~\eqref{ineq:g-h-xk} and~\eqref{ineq:psi-xk}, we obtain
        \begin{align}
            \label{ineq:psi-nkxk}
            |\Psi(x^k) - \zeta|^{\theta}\leq\nu\left\|\nabla \left(g + \frac{1}{\lambda}h\right)(x^k) - \nabla \left(g + \frac{1}{\lambda}h\right)(x^{k+1})\right\|
            \leq \kappa\nu\|x^k-x^{k+1}\|, \quad \forall k \geq \bar{k}.
        \end{align}
        Since the function $t^{1-\theta}$ is concave on $[0, \infty)$, $\Psi(x^k) - \zeta \geq 0$,
        \eqref{ineq:sufficient-decrease}, and~\eqref{ineq:psi-nkxk},
        we find that, for all $k \geq \bar{k}$,
        \begin{align}
        (\Psi(x^k) - \zeta)^{1-\theta} - (\Psi(x^{k+1}) - \zeta)^{1-\theta}
        &{}\geq (1-\theta)(\Psi(x^k) - \zeta)^{-\theta}(\Psi(x^k) - \Psi(x^{k+1}))\nonumber\\
        &{}\geq \frac{1-\theta}{\kappa\nu\|x^k-x^{k+1}\|}\left( \frac{1}{\lambda} - L \right)\frac{\sigma}{2}\|x^k-x^{k+1}\|^2\nonumber\\
        &{}= \frac{(1-\theta)\sigma}{2\kappa\nu}\left(\frac{1}{\lambda} - L \right)\|x^k-x^{k+1}\|.\label{ineq:psi-zeta-xk}
        \end{align}
        Summing~\eqref{ineq:psi-zeta-xk} from $k=\bar{k}$ to $\infty$ yields
        \begin{align*}
            \sum_{k=\bar{k}}^{\infty}\|x^k-x^{k+1}\|\leq\frac{2\kappa\nu}{(1/\lambda - L)(1-\theta)\sigma}(\Psi(x^{\bar{k}}) - \zeta)^{1-\theta} < \infty,
        \end{align*}
        which implies that $\sum_{k=1}^{\infty} \|x^k - x^{k-1}\| < \infty$, \ie,
        the sequence $\{x^k\}_{k=0}^{\infty}$ is a Cauchy sequence.
        Thus, $\{x^k\}_{k=0}^{\infty}$ converges to a limiting critical point of $(\mathcal{P})$ from Theorem~\ref{theorem:global-subsequential-conv-bpdca} (iv).
        \qed
    \end{proof}

    Finally, we will show the rate of convergence
    in a manner following~\cite{attouch09,pdcae}.

    \begin{theorem}[Rate of convergence under the local differentiability of $f_2$]
    \label{theorem:rate-bpdca}
    Suppose that Assumptions~$\ref{ass1}$,~$\ref{ass2}$,~$\ref{ass4}$, and~$\ref{ass5}$ hold.
    Let $\{x^k\}_{k=0}^{\infty}$ be a sequence generated by BPDCA with $0 < \lambda L < 1$ for solving $(\mathcal{P})$
    and suppose that $\{x^k\}_{k=0}^{\infty}$ converges to some $\tilde{x}\in\mathcal{X}$.
    Suppose further that $\Psi$ is a KL function with $\phi$ in the KL inequality~$\eqref{ineq:kl}$
    taking the form $\phi(s) = cs^{1-\theta}$ for some $\theta \in [0, 1)$ and $c > 0$.
    Then, the following statements hold:
    \begin{enumerate}
        \renewcommand{\labelenumi}{\rm{(\roman{enumi})}}
        \item If $\theta = 0$, then there exists $k_0 > 0$ such that $x^k$ is constant for $k > k_0$;
        \item If $\theta \in (0, \frac{1}{2}]$, then there exist $c_1 > 0, k_1 > 0$, and $\eta\in(0, 1)$
        such that $\|x^k - \tilde{x}\| < c_1\eta^k$ for $k > k_1$;
        \item If $\theta \in (\frac{1}{2}, 1)$, then there exist $c_2 > 0$ and $k_2 > 0$
        such that $\|x^k - \tilde{x}\| < c_{2}k^{-\frac{1-\theta}{2\theta-1}}$ for $k > k_2$.
    \end{enumerate}
    \end{theorem}

    \begin{proof}
    (i)
        For the case of $\theta = 0$, we will prove that there exists an integer $k_0 > 0$ such that $\Psi(x^{k_0}) = \zeta$
        by assuming to the contrary that $\Psi(x^k) > \zeta$ for all $k > 0$ and showing a contradiction.
        The sequence $\{\Psi(x^k)\}_{k=0}^{\infty}$ converges to $\zeta$ due to Proposition~\ref{prop:zeta} (i).
        In addition, from the KL inequality~\eqref{ineq:phi-kl-iter} and $\phi'(\cdot) = c$,
        we can see that for all sufficiently large $k$,
        \begin{align*}
            \dist(0, \partial\Psi(x^k)) \geq \frac{1}{c},
        \end{align*}
        which contradicts Theorem~\ref{theorem:global-convergence-bpdca} (i).
        Therefore, there exists $k_0 > 0$ such that $\Psi(x^{k_0}) = \zeta$.
        Since $\{\Psi(x^k)\}_{k=0}^{\infty}$ is non-increasing and converges to $\zeta$,
        we have $\Psi(x^{k_0+\bar{k}}) = \zeta$ for all $\bar{k} \geq 0$.
        This, together with~\eqref{ineq:sufficient-decrease-rho},
        lead us to conclude that there exists $k_0 > 0$ such that $x^k$ is constant for $k > k_0$.

        (ii--iii)
        Next, consider the case $\theta\in(0, 1)$.
        If there exists $k_0 > 0$ such that $\Psi(x^{k_0}) = \zeta$,
        then we can show that the sequence $\{x^k\}_{k=0}^{\infty}$ is finitely convergent in the same way as in the proof of (i).
        Therefore, for $\theta\in(0, 1)$, we only need to consider the case that $\Psi(x^k) > \zeta$ for all $k > 0$.

        Define $R_k = \Psi(x^k) - \zeta$ and $S_k = \sum_{j=k}^{\infty}\|x^{j+1} - x^j\|$, where $S_k$ is well-defined
        due to Theorem~\ref{theorem:global-convergence-bpdca} (ii).
        From~\eqref{ineq:summing}, for any $k \geq \bar{k}$, where $\bar{k}$ is defined in~\eqref{ineq:phi-kl-iter},
        we obtain
        \begin{align}
            S_k &= 2\sum_{j=k}^{\infty}\frac{1}{2}\|x^{j+1}-x^j\|\nonumber\\
            &\leq2\sum_{j=k}^{\infty}\left[ \frac{A_1}{2A_2} \left( \phi(\Psi(x^j) - \zeta) - \phi(\Psi(x^{j+1}) - \zeta) \right)
            + \frac{1}{2}\|x^j - x^{j-1}\| - \frac{1}{2}\|x^{j+1} - x^j \|\right]\nonumber\\
            &\leq\frac{A_1}{A_2}\phi(\Psi(x^k) - \zeta) + \|x^k - x^{k-1}\|\nonumber\\
            &= \frac{A_1}{A_2}\phi(R_k) + S_{k-1} - S_k. \label{ineq:norm-seq}
        \end{align}
        On the other hand, since $\lim_{k\to\infty}x^k = \tilde{x}$ and $\{\Psi(x^k)\}$ is non-increasing and converges to $\zeta$,
        the KL inequality~\eqref{ineq:phi-kl-iter} with $\phi(s) = cs^{1-\theta}$ ensures that, for all sufficiently large $k$,
        \begin{align}
            \label{ineq:phi-kl-psi}
            c(1-\theta)R_k^{-\theta}\dist(0, \partial\Psi(x^k))\geq1.
        \end{align}
        From the definition of $S_k$ and~\eqref{ineq:dist-euc}, we also have that, for all sufficiently large $k$,
        \begin{align}
            \label{ineq:dist-partial-norm-seq}
            \dist(0, \partial\Psi(x^k)) \leq A_1(S_{k-1} - S_k).
        \end{align}
        Combining~\eqref{ineq:phi-kl-psi} and~\eqref{ineq:dist-partial-norm-seq},
        we have $R_k^{\theta} \leq A_1 \cdot c(1-\theta) (S_{k-1} - S_k)$ for all sufficiently large $k$.
        Raising the above inequality to the power of $\frac{1-\theta}{\theta}$ and scaling both sides by $c$,
        we find that $cR_k^{1-\theta} \leq c(A_1 \cdot c(1-\theta) (S_{k-1} - S_k))^{\frac{1-\theta}{\theta}}$.
        Combining this with~\eqref{ineq:norm-seq} and recalling $\phi(R_k) = cR_k^{1-\theta}$, we find that,
        for all sufficiently large $k$,
        \begin{align}
            \label{ineq:sk-2}
            S_k \leq A_3 (S_{k-1} - S_k)^{\frac{1-\theta}{\theta}} + S_{k-1} - S_k,
        \end{align}
        where $A_3 = \frac{A_1}{A_2}c(A_1 \cdot c(1-\theta))^{\frac{1-\theta}{\theta}}$.

        (ii)
        When $\theta\in(0,\frac{1}{2}]$, we have $\frac{1-\theta}{\theta}\geq1$.
        Since $\lim_{k\to\infty}\|x^{k+1} - x^k\| = 0$ by Theorem~\ref{theorem:global-subsequential-conv-bpdca} (iii),
        $\lim_{k\to\infty}S_{k-1} - S_k = 0$.
        From these considerations and~\eqref{ineq:sk-2}, we conclude that there exists $k_1 > 0$ such that for all $k \geq k_1$,
        $S_k \leq (A_3 + 1)(S_{k-1} - S_k)$,
        which implies $S_k \leq \frac{A_3 + 1}{A_3 + 2} S_{k-1}$.
        Therefore, for all $k \geq k_1$,
        \begin{align*}
            \|x^k - \tilde{x}\| \leq \sum_{j=k}^{\infty}\|x^{j+1} - x^j\| = S_k \leq S_{k_1-1}\left( \frac{A_3 + 1}{A_3 + 2} \right)^{k - k_1 + 1}.
        \end{align*}

        (iii)
        For $\theta\in\left(\frac{1}{2}, 1\right)$, $\frac{1-\theta}{\theta} < 1$.
        From~\eqref{ineq:sk-2} and $\lim_{k\to\infty}S_{k-1} - S_k = 0$,
        there exists $k_2 > 0$ such that
        \begin{align*}
            S_k \leq A_3 (S_{k-1} - S_k)^{\frac{1-\theta}{\theta}} + S_{k-1} - S_k
            &\leq A_3 (S_{k-1} - S_k)^{\frac{1-\theta}{\theta}} + (S_{k-1} - S_k)^{\frac{1-\theta}{\theta}}\\
            &\leq (A_3 + 1)(S_{k-1} - S_k)^{\frac{1-\theta}{\theta}},
        \end{align*}
        for all $k \geq k_2$.
        Raising the above inequality to the power of $\frac{\theta}{1-\theta}$, for any $k \geq k_2$ we find that
        $S_k^{\frac{\theta}{1-\theta}} \leq A_4(S_{k-1} - S_k)$,
        where $A_4 = (A_3 + 1)^{\frac{\theta}{1-\theta}}$.
        From~\cite[Theorem 2]{attouch09}, we find that, for all sufficiently large $k$,
        there exists $A_5 > 0$ such that $S_k \leq A_5 k^{-\frac{1 - \theta}{2\theta - 1}}$.
        \qed
    \end{proof}
    Using Theorem~\ref{theorem:global-convergence-bpdca-g},
    we can obtain another rate of convergence
    in the same way as in the proof of~\cite[Theorem 2]{attouch09} or~\cite[Theorem 3.5]{hoai2018}.
    \begin{theorem}[Rate of convergence under the local differentiability of $g$]
    \label{theorem:rate-bpdca-g}
    Suppose that Assumptions~$\ref{ass1}$,~$\ref{ass2}$,~$\ref{ass4}$, and~$\ref{ass5g}$ hold.
    Let $\{x^k\}_{k=0}^{\infty}$ be a sequence generated by BPDCA with $0 < \lambda L < 1$ for solving $(\mathcal{P})$
    and suppose that $\{x^k\}_{k=0}^{\infty}$ converges to some $\tilde{x}\in\mathcal{X}$.
    Suppose further that $\Psi$ is subanalytic.
    Let $\theta \in [0, 1)$ be a \L ojasiewicz exponent of $\tilde{x}$.
    Then, the following statements hold:
    \begin{enumerate}
        \renewcommand{\labelenumi}{\rm{(\roman{enumi})}}
        \item If $\theta = 0$, then there exists $k_0 > 0$ such that $x^k$ is constant for $k > k_0$;
        \item If $\theta \in (0, \frac{1}{2}]$, then there exist $c_1 > 0, k_1 > 0$, and $\eta\in(0, 1)$
        such that $\|x^k - \tilde{x}\| < c_1\eta^k$ for $k > k_1$;
        \item If $\theta \in (\frac{1}{2}, 1)$, then there exist $c_2 > 0$ and $k_2 > 0$
        such that $\|x^k - \tilde{x}\| < c_{2}k^{-\frac{1-\theta}{2\theta-1}}$ for $k > k_2$.
    \end{enumerate}
    \end{theorem}

    \subsection{Properties of BPDCAe}\label{subsec:properties-of-bpdcae}
    Inspired by~\cite{bpge}, we introduce the auxiliary function,
    \begin{align*}
        H_M(x, y) = \Psi(x) + MD_h(y, x),\quad M>0.
    \end{align*}
    To show the decreasing property of $H_M$, instead of $\Psi$,
    with respect to $\{x^k\}_{k=0}^{\infty}$,
    we further assume the convexity of $g$.
    \begin{assumption}
        \label{ass:g-convex}
        The function $g$ is convex.
    \end{assumption}
    Under the adaptive restart scheme (see~\eqref{adaptive-restart}),
    we show the decreasing property of $H_M$.
    \begin{lemma}
        \label{lemma:obj-decrease-ex}
        Suppose that Assumptions~$\ref{ass1}$,~$\ref{ass2-ex}$, and~$\ref{ass:g-convex}$ hold.
        For any $x^k, y^k \in C = \interior\dom h$ and any $x^{k+1} \in C =\interior\dom h$ defined by
        \begin{align}
            \label{eq:conv-update-ex}
            x^{k+1} \in \argmin_{y \in \real^d} \left\{ g(y) + \langle \nabla f_1(y^k) - \xi^k, y - y^k \rangle + \frac{1}{\lambda}D_h(y, y^k) \right\},
        \end{align}
        where $\xi^k \in \partial_{\mathrm{c}} f_2(x^k)$, $y^k = x^k + \beta_k(x^k - x^{k-1}), \lambda > 0$, and $\{\beta_k\}_{k=0}^{\infty}\subset [0,1)$,
        it holds that
        \begin{align}
            \label{ineq:obj-decrease-ex}
            \lambda \Psi(x^{k+1}) \leq \lambda \Psi(x^k) + D_h(x^k, y^k) - D_h(x^k, x^{k+1}) - (1 - \lambda L) D_h(x^{k+1}, y^k).
        \end{align}
        Furthermore, when $0 < \lambda L < 1$ and $\{\beta_k\}_{k=0}^\infty$ is given by the adaptive restart scheme~\eqref{adaptive-restart},
        \begin{align}
            \label{ineq:aux-decrease-ex}
            H_M(x^{k+1}, x^{k})
            \leq{}& H_M(x^k, x^{k-1})
            - \left(\frac{1}{\lambda} - M \right)D_h(x^k, x^{k+1})\nonumber\\
            &- \left( M - \frac{\rho}{\lambda} \right) D_h(x^{k-1}, x^k)
            - \left(\frac{1}{\lambda}-L\right)D_h(x^{k+1},y^k).
        \end{align}
        In addition, when $\frac{\rho}{\lambda} \leq M \leq \frac{1}{\lambda}$ for $\rho \in [0, 1)$,
        the auxiliary function $H_M$ is ensured to be non-increasing.
    \end{lemma}

    \begin{proof}
        From the first-order optimality condition for~\eqref{eq:conv-update-ex}, we obtain
        \begin{align*}
            0 \in \partial_{\mathrm{c}} g(x^{k+1}) + \nabla f_1(y^k) - \xi^k + \frac{1}{\lambda} (\nabla h(x^{k+1}) - \nabla h(y^k)).
        \end{align*}
        From the convexity of $g$, we find that
        \begin{align*}
            g(x^k) - g(x^{k+1}) \geq \langle -\nabla f_1(y^k) + \xi^k - \frac{1}{\lambda} (\nabla h(x^{k+1}) - \nabla h(y^k)), x^k - x^{k+1} \rangle.
        \end{align*}
        Using the three-point identity~\eqref{eq:three-point} of the Bregman distances,
        \begin{align*}
            \frac{1}{\lambda}\langle \nabla h(x^{k+1}) - \nabla h(y^k), x^k - x^{k+1} \rangle = \frac{1}{\lambda} (D_h(x^k, y^k) - D_h(x^k, x^{k+1}) - D_h(x^{k+1}, y^k)),
        \end{align*}
        we have
        \begin{align*}
            g(x^k) - g(x^{k+1}) + f_1(x^k) - f_1(x^{k+1})
            \geq{}& f_1(x^k) - f_1(x^{k+1}) + \langle - \nabla f_1(y^k) + \xi^k, x^k - x^{k+1} \rangle \\
            &-\frac{1}{\lambda} \left(D_h(x^k, y^k) - D_h(x^k, x^{k+1}) - D_h(x^{k+1}, y^k)\right).
        \end{align*}
        From the convexity of $f_1$ and Lemma~\ref{lemma:descent-lemma}, we find that
        \begin{align*}
            &f_1(x^k) - f_1(x^{k+1}) - \langle \nabla f_1(y^k), x^k - x^{k+1} \rangle\\
            ={} &f_1(x^k) - f_1(y^k) - \langle \nabla f_1(y^k), x^k - y^k \rangle  - f_1(x^{k+1}) + f_1(y^k) + \langle \nabla f_1(y^k), x^{k+1} - y^k \rangle\\
            \geq{}&\!-LD_h(x^{k+1}, y^k).
        \end{align*}
        The above inequalities and the definition of the subgradient for $f_2$ lead us to
        \begin{align*}
            \Psi(x^{k+1}) \leq \Psi(x^k) + \frac{1}{\lambda}D_h(x^k, y^k) - \frac{1}{\lambda}D_h(x^k, x^{k+1}) - \left(\frac{1}{\lambda} -L\right) D_h(x^{k+1}, y^k),
        \end{align*}
        which implies inequality~\eqref{ineq:obj-decrease-ex}.
        If $\beta_k = 0$, then $y^k = x^k$ and $D_h(x^k, y^k) = 0$.
        If $\beta_k \neq 0$, since we chose the adaptive restart scheme, there is a $\rho \in [0, 1)$ satisfying $D_h(x^k, y^k) \leq \rho D_h(x^{k-1}, x^k)$.
        From the definition of $H_M(x^k, x^{k-1})$ and $0 < \lambda L < 1$, we have
        \begin{align}
            H_M(x^{k+1}, x^{k})
            \leq{}& H_M(x^k, x^{k-1})
            + \frac{1}{\lambda}D_h(x^k, y^k)
            - \left(\frac{1}{\lambda} - M \right)D_h(x^k, x^{k+1})\nonumber\\
            &- M D_h(x^{k-1}, x^k)
            - \left(\frac{1}{\lambda}-L\right)D_h(x^{k+1},y^k)\nonumber\\
            \leq{}& H_M(x^k, x^{k-1})
            - \left(\frac{1}{\lambda} - M \right)D_h(x^k, x^{k+1})\nonumber\\
            &- \left( M - \frac{\rho}{\lambda} \right) D_h(x^{k-1}, x^k)
            - \left(\frac{1}{\lambda}-L\right)D_h(x^{k+1},y^k),\label{ineq:aux-decrease}
        \end{align}
        where the second inequality comes from~$D_h(x^k, y^k) \leq \rho D_h(x^{k-1}, x^k)$.
        When $\frac{\rho}{\lambda} \leq M \leq \frac{1}{\lambda}$, we have
        \begin{align*}
            H_M(x^{k+1}, x^{k}) \leq H_M(x^k, x^{k-1}), \quad \forall k \geq 0,
        \end{align*}
        which shows that the sequence $\{H_M\}_{k=0}^{\infty}$ is non-increasing.
        \qed
    \end{proof}

    We can use Lemma~\ref{lemma:obj-decrease-ex} to prove Proposition~\ref{prop:bpdcae-property}.

    \begin{proposition}
        \label{prop:bpdcae-property}
        Suppose that Assumptions~$\ref{ass1}$,~$\ref{ass2-ex}$, and~$\ref{ass:g-convex}$ hold.
        Let $\{x^k\}_{k=0}^{\infty}$ be a sequence generated by BPDCAe with $0 < \lambda L < 1$.
        Assume that the auxiliary function $H_M(x^k, x^{k-1})$ satisfies $\frac{\rho}{\lambda} \leq M \leq \frac{1}{\lambda}$ for
        $\rho \in [0, 1)$.
        Then, the following statements hold:
        \begin{enumerate}
            \renewcommand{\labelenumi}{\rm{(\roman{enumi})}}
            \item The sequence $\{H_M(x^k, x^{k-1})\}_{k=0}^{\infty}$ is non-increasing.
            \item $\sum_{k=1}^{\infty}D_h(x^{k-1}, x^{k}) < \infty$; hence, the sequence $\{D_h(x^{k-1}, x^{k})\}_{k=0}^{\infty}$ converges to zero.
            \item $\min_{1\leq k \leq n}D_h(x^{k-1}, x^{k}) \leq \frac{\lambda}{n(1-\rho)}\left( \Psi(x^0) - \Psi_* \right)$,
            where $\Psi_* = v(\mathcal{P}) > -\infty$~\rm{(by Assumption~\ref{ass1} (v))}.
        \end{enumerate}
    \end{proposition}
    \begin{proof}
    (i) The statement was proved in Lemma~\ref{lemma:obj-decrease-ex}.

    (ii) Modify~\eqref{ineq:aux-decrease} into
    \begin{align*}
        \lambda( H_M(x^{k+1}, x^{k}) - H_M(x^{k}, x^{k-1}))
        \leq& -(1 - \lambda M) D_h(x^k, x^{k+1}) - (\lambda M - \rho)D_h(x^{k-1}, x^{k})\\
        &- (1-\lambda L) D_h(x^{k+1}, y^k)\\
        \leq& -(1 - \lambda M) D_h(x^k, x^{k+1}) - (\lambda M - \rho)D_h(x^{k-1}, x^{k}),
    \end{align*}
    where the last inequality comes from $(1-\lambda L) D_h(x^{k+1}, y^k) \geq 0$.
    Let $n$ be a positive integer.
    Summing the above inequality from $k=0$ to $n$ and letting $\Psi_* = v(\mathcal{P}) > - \infty$,
    we find that
    \begin{align}
        \sum_{k=1}^n D_h(x^{k-1}, x^{k}) = \sum_{k=0}^n D_h(x^{k-1}, x^{k})
        &\leq \frac{\lambda\left( H_M(x^0, x^{-1}) - H_M(x^{n+1}, x^{n}) \right)}{1 - \rho}\nonumber\\
        &\leq \frac{\lambda\left( \Psi(x^0) - \Psi(x^{n+1}) \right)}{1 - \rho}\nonumber\\
        &\leq \frac{\lambda\left( \Psi(x^0) - \Psi_* \right)}{1 - \rho},\label{ineq:aux-property-bregman-sum}
    \end{align}
    where the second inequality comes from $D_h(x^{-1}, x^0) = 0$, $x^{-1} = x^0$, and $D_h(x^{n}, x^{n+1}) \geq 0$.
    Note that $x^{n+1} \in C$ by Assumption~\ref{ass2-ex}.
    By taking the limit as $n \to \infty$, we arrive at the former statement (ii).
    The latter statement directly follows from the former.

    (iii)
    From~\eqref{ineq:aux-property-bregman-sum}, we immediately have
    \begin{align*}
        n\min_{1\leq k \leq n}D_h(x^{k-1}, x^{k})
        \leq \sum_{k=1}^n D_h(x^{k-1}, x^{k})
        \leq \frac{\lambda\left( \Psi(x^0) - \Psi_*\right)}{1 - \rho}.
    \end{align*}
    \qed
    \end{proof}

    \subsection{Convergence Analysis of BPDCAe}\label{subsec:convergence-analysis-of-bpdcae}
    The proofs of Theorems~\ref{theorem:global-subsequential-conv-bpdcae},~\ref{theorem:global-convergence-bpdcae},~\ref{theorem:global-convergence-bpdcae-g},
    and Proposition~\ref{prop:zeta-ex} are given in the Appendix.
    They follow arguments that are similar to their BPDCA counterparts.
    \begin{theorem}[Global subsequential convergence of BPDCAe]
        \label{theorem:global-subsequential-conv-bpdcae}
        Suppose that Assumptions $\ref{ass1}$, $\ref{ass2-ex}$, $\ref{ass4}$, and $\ref{ass:g-convex}$ hold.
        Let $\{x^k\}_{k=0}^{\infty}$ be a sequence generated by BPDCAe with $0 < \lambda L < 1$ for solving $(\mathcal{P})$.
        Assume that the auxiliary function $H_M(x^k, x^{k-1})$
        satisfies $\frac{\rho}{\lambda} \leq M \leq \frac{1}{\lambda} $
        for $\rho \in [0, 1)$.
        Then, the following statements hold:
        \begin{enumerate}
            \renewcommand{\labelenumi}{\rm{(\roman{enumi})}}
            \item The sequence $\{x^k\}_{k=0}^{\infty}$ is bounded.
            \item $\lim_{k\to\infty}\|x^{k+1} - x^k\| = 0$.
            \item Any accumulation point of $\{x^k\}_{k=0}^{\infty}$ is a limiting critical point of $(\mathcal{P})$.
        \end{enumerate}
    \end{theorem}
    \begin{proposition}
        \label{prop:zeta-ex}
        Suppose that Assumptions $\ref{ass1}$, $\ref{ass2-ex}$, $\ref{ass4}$, and $\ref{ass:g-convex}$ hold.
        Let $\{x^k\}_{k=0}^{\infty}$ be a sequence generated by BPDCAe with $0 < \lambda L < 1$ for solving $(\mathcal{P})$
        and $\frac{\rho}{\lambda} \leq M \leq \frac{1}{\lambda}$ for $\rho \in [0, 1)$.
        Then, the following statements hold:
        \begin{enumerate}
            \renewcommand{\labelenumi}{\rm{(\roman{enumi})}}
            \item $\zeta := \lim_{k\to\infty}\Psi(x^k)$ exists.
            \item $\Psi \equiv \zeta$ on $\Omega$, where $\Omega$ is the set of accumulation points of $\{x^k\}_{k=0}^{\infty}$.
        \end{enumerate}
    \end{proposition}

    Since $H_M(x, y)$ has a Bregman distance term,
    the subdifferential of $H_M(x, y)$ has a $\nabla h$ term.
    To prove Theorem~\ref{theorem:global-convergence-bpdcae}, we should additionally suppose
    that there is a bounded subdifferential of the gradient $\nabla h$~\cite{bpge}.

    \begin{assumption}
        \label{ass6}
        There exists a bounded $u$ such that $u \in \partial (\nabla h)$ on any bounded subset of $\real^d$.
    \end{assumption}

    We can prove the following theorems by supposing the KL property or the subanalyticity of
    the auxiliary function $H_M(x, y)$ in relation to $x$ and $y$.
    \begin{theorem}[Global convergence of BPDCAe under the local differentiability of $f_2$]
    \label{theorem:global-convergence-bpdcae}
    Suppose that Assumptions $\ref{ass1}$, $\ref{ass2-ex}$, $\ref{ass4}$, $\ref{ass5}$, $\ref{ass:g-convex}$, and $\ref{ass6}$ hold
    and that the auxiliary function $H_M(x, y)$ is a KL function satisfying $\frac{\rho}{\lambda} \leq M \leq \frac{1}{\lambda}$ for $\rho \in [0, 1)$.
    Let $\{x^k\}_{k=0}^{\infty}$ be a sequence generated by BPDCAe with $0 < \lambda L < 1$ for solving $(\mathcal{P})$.
    Then, the following statements hold:
    \begin{enumerate}
        \renewcommand{\labelenumi}{\rm{(\roman{enumi})}}
        \item $\lim_{k\to\infty}\dist((0, 0), \partial H_M(x^k, x^{k-1})) = 0$.
        \item The set of accumulation points of $\{(x^k, x^{k-1})\}_{k=0}^{\infty}$ is $\Upsilon := \left\{ (x, x)\ \middle|\ x\in\Omega \right\}$
        and $H_M\equiv\zeta$ on $\Upsilon$, where $\Omega$ is the set of accumulation points of $\{x^k\}_{k=0}^{\infty}$.
        \item The sequence $\{x^k\}_{k=0}^{\infty}$ converges to a limiting stationary point of $(\mathcal{P})$;
        moreover, $\sum_{k=1}^{\infty} \|x^k - x^{k-1}\| < \infty$.
    \end{enumerate}
    \end{theorem}
    \begin{theorem}[Global convergence of BPDCAe under the local differentiability of $g$]
    Suppose that Assumptions $\ref{ass1}$, $\ref{ass2-ex}$, $\ref{ass4}$, $\ref{ass5g}$, $\ref{ass:g-convex}$, and $\ref{ass6}$ hold
    and that the auxiliary function $H_M(x, y)$ is subanalytic satisfying $\frac{\rho}{\lambda} \leq M \leq \frac{1}{\lambda}$ for $\rho \in [0, 1)$.
    Let $\{x^k\}_{k=0}^{\infty}$ be a sequence generated by BPDCAe with $0 < \lambda L < 1$ for solving $(\mathcal{P})$.
    Then, the sequence $\{x^k\}_{k=0}^{\infty}$ converges to a limiting critical point of $(\mathcal{P})$;
    moreover, $\sum_{k=1}^{\infty} \|x^k - x^{k-1}\| < \infty$.
    \label{theorem:global-convergence-bpdcae-g}
    \end{theorem}
    Finally, we have theorems regarding the convergence rate of BPDCAe, whose proof is almost identical to Theorems~\ref{theorem:rate-bpdca} and~\ref{theorem:rate-bpdca-g}.
    \begin{theorem}[Rate of convergence under the local differentiability of $f_2$]
    Suppose that Assumptions $\ref{ass1}$, $\ref{ass2-ex}$, $\ref{ass4}$, $\ref{ass5}$, $\ref{ass:g-convex}$, and $\ref{ass6}$ hold.
    Let $\{x^k\}_{k=0}^{\infty}$ be a sequence generated by BPDCAe with $0 < \lambda L < 1$ for solving $(\mathcal{P})$
    and suppose that $\{x^k\}_{k=0}^{\infty}$ converges to some $\tilde{x}\in\mathcal{X}$.
    Suppose further that the auxiliary function $H_M(x, y)$ satisfying $\frac{\rho}{\lambda} \leq M \leq \frac{1}{\lambda}$ for $\rho \in [0, 1)$ is a KL function with $\phi$ in the KL inequality~$\eqref{ineq:kl}$
    taking the form $\phi(s) = cs^{1-\theta}$ for some $\theta \in [0, 1)$ and $c > 0$.
    Then, the following statements hold:
    \begin{enumerate}
        \renewcommand{\labelenumi}{\rm{(\roman{enumi})}}
        \item If $\theta = 0$, then there exists $k_0 > 0$ such that $x^k$ is constant for $k > k_0$;
        \item If $\theta \in (0, \frac{1}{2}]$, then there exist $c_1 > 0, k_1 > 0$, and $\eta\in(0, 1)$
        such that $\|x^k - \tilde{x}\| < c_1\eta^k$ for $k > k_1$;
        \item If $\theta \in (\frac{1}{2}, 1)$, then there exist $c_2 > 0$ and $k_2 > 0$
        such that $\|x^k - \tilde{x}\| < c_{2}k^{-\frac{1-\theta}{2\theta-1}}$ for $k > k_2$.
    \end{enumerate}
    \end{theorem}
    \begin{theorem}[Rate of convergence under the local differentiability of $g$]
    \label{theorem:rate-bpdcae-g}
    Suppose that Assumptions $\ref{ass1}$, $\ref{ass2-ex}$, $\ref{ass4}$, $\ref{ass5g}$, $\ref{ass:g-convex}$, and $\ref{ass6}$ hold.
    Let $\{x^k\}_{k=0}^{\infty}$ be a sequence generated by BPDCAe with $0 < \lambda L < 1$ for solving $(\mathcal{P})$
    and suppose that $\{x^k\}_{k=0}^{\infty}$ converges to some $\tilde{x}\in\mathcal{X}$.
    Suppose further that the auxiliary function $H_M(x, y)$ satisfying $\frac{\rho}{\lambda} \leq M \leq \frac{1}{\lambda}$ for $\rho \in [0, 1)$ is subanalytic.
    Let $\theta \in [0, 1)$ be a \L ojasiewicz exponent of $\tilde{x}$.
    Then, the following statements hold:
    \begin{enumerate}
        \renewcommand{\labelenumi}{\rm{(\roman{enumi})}}
        \item If $\theta = 0$, then there exists $k_0 > 0$ such that $x^k$ is constant for $k > k_0$;
        \item If $\theta \in (0, \frac{1}{2}]$, then there exist $c_1 > 0, k_1 > 0$, and $\eta\in(0, 1)$
        such that $\|x^k - \tilde{x}\| < c_1\eta^k$ for $k > k_1$;
        \item If $\theta \in (\frac{1}{2}, 1)$, then there exist $c_2 > 0$ and $k_2 > 0$
        such that $\|x^k - \tilde{x}\| < c_{2}k^{-\frac{1-\theta}{2\theta-1}}$ for $k > k_2$.
    \end{enumerate}
    \end{theorem}

    \section{Applications}\label{sec:applications}

    \subsection{Application to Phase Retrieval}\label{subsec:application-to-quadratic-inverse-problems}
    In phase retrieval, we are interested in finding a (parameter) vector $x\in\real^d$ that approximately solves the system,
    \begin{align}
        \label{prob:phase}
        \langle a_r, x \rangle^2 \simeq b_r, \quad r=1, 2,\ldots,m,
    \end{align}
    where the vectors $a_r\in\real^d$ describe the model and $b=(b_1,b_2,\ldots,b_m)^{\mathrm{T}}$ is a vector of (usually) noisy measurements.
    As described in~\cite{bpg,wirtinger}, the system~$\eqref{prob:phase}$ can be formulated as a nonconvex optimization problem:
    \begin{align}
        \label{prob:pr}
        \min \left\{ \Psi(x) := \frac{1}{4}\sum_{r=1}^m\left( \langle a_r, x \rangle^2 - b_r \right)^2 + \theta g(x)\ \middle|\ x \in \real^d \right\},
    \end{align}
    where $\theta \geq 0$ is a trade-off parameter between the data fidelity criteria and the regularizer $g$.
    We define $g:\real^d \to \real$, in particular $g(x) = \|x\|_1$.

    In this case, the underlying space of $(\mathcal{P})$ is $C \equiv \real^d$.
    Define $f:\real^d\to\real$ as $f(x) = \frac{1}{4} \sum_{r=1}^m \left( \langle a_r, x \rangle^2 - b_r \right)^2$, which is
    a nonconvex differentiable function that does not admit a global Lipschitz continuous gradient.
    The objective function of the phase retrieval problem can be also reformulated as a difference between two convex functions such as in~\cite{Huang2018}.
    That is, $f(x) = f_1(x) - f_2(x)$, where
    \begin{align}
        \label{func:dc-real}
        f_1(x) = \frac{1}{4} \sum_{r=1}^m \langle a_r, x \rangle^4 + \frac{1}{4}\|b\|^2 \quad \mathrm{and} \quad
        f_2(x) = \frac{1}{2} \sum_{r=1}^m b_r \langle a_r, x \rangle^2.
    \end{align}
    When we do not regard the phase retrieval~\eqref{prob:pr} as a DC optimization problem,
    the Bregman Proximal Gradient algorithm (BPG) can be used instead~\cite{bpg}.
    Enhancements using the extrapolation technique were proposed:
    the Bregman Proximal Gradient algorithm with extrapolation (BPGe)~\cite{bpge} and
    Convex-Concave Inertial BPG~\cite{Mukkamala2020} for estimating $L$.
    For BPG(e),
    assuming $L$-smad for the pair $(f_1 - f_2, h)$ using $h(x) = \frac{1}{4}\|x\|^4 + \frac{1}{2}\|x\|^2$,
    $L$ satisfies the following inequality~\cite[Lemma 5.1]{bpg}:
    \begin{align}
        \label{ineq:L-bpg}
        L \geq \sum_{r=1}^m \left( 3\|a_r a_r^{\mathrm{T}}\|^2 + \|a_r a_r^{\mathrm{T}}\||b_r| \right).
    \end{align}

    On the other hand, for DC optimization problems,
    we define $h:\real^d\to\real$ as
    \begin{align}
        \label{func:ker-real}
        h(x) = \frac{1}{4}\|x\|^4.
    \end{align}
    This function is simpler than the original nonconvex formulation.
    The function $h(x) = \frac{1}{4} \|x\|^4$ is not $\sigma$-strongly convex.
    Therefore, this function does not satisfy Assumption~\ref{ass4} (i).

    \begin{proposition}
        \label{prop:L-bpdc-sum}
        Let $f_1$ and $h$ be as defined above.
        Then, for any $L$ satisfying
        \begin{align}
            \label{ineq:L-bpdc-s}
            L \geq 3\left\|\sum_{r=1}^m \|a_r\|^2 a_r a_r^{\mathrm{T}}\right\|,
        \end{align}
        the function $Lh-f_1$ is convex on $\real^d$.
        Therefore, the pair $(f_1,h)$ is L-smad on $\real^d$.
    \end{proposition}

    \begin{proof}
        Let $x\in\real^d$.
        Since $f_1$ and $h$ are $\mathcal{C}^2$ on $\real^d$,
        to guarantee the convexity of $Lh-f_1$,
        it is sufficient to find $L > 0$ such that $L\lambda_{\min}(\nabla^2 h(x)) \geq \lambda_{\max}(\nabla^2 f_1(x))$,
        where $\lambda_{\min}(M)$ and $\lambda_{\max}(M)$ denote the minimal and maximal eigenvalues of a matrix $M$, respectively.
        Now, we have the Hessian for $f_1$ and $h$:
        \begin{align*}
            \nabla^2 f_1(x)
            = 3 \sum_{r=1}^m \langle a_r, x \rangle^2 a_r a_r^{\mathrm{T}}
            \quad \mathrm{and} \quad
            \nabla^2 h(x) = \|x\|^2 I_d + 2xx^{\mathrm{T}}.
        \end{align*}
        Since $\nabla^2 h(x) \succeq \|x\|^2 I_d$, we obtain $\lambda_{\min}\left(\nabla^2 h(x)\right) \geq \|x\|^2$.
        From the well-known fact, $\lambda_{\max}(M) \leq\|M\|$,
        we have the following inequality:
        \begin{align*}
            \lambda_{\max}\left(\nabla^2 f_1(x)\right)
            \leq 3\left\| \sum_{r=1}^m \langle a_r, x \rangle^2 a_r a_r^{\mathrm{T}} \right\|
            \leq 3\left\| \sum_{r=1}^m\|a_r\|^2 a_r a_r^{\mathrm{T}} \right\|\|x\|^2
            \leq L\|x\|^2 \leq L\lambda_{\min}(\nabla^2 h(x)).
        \end{align*}
        Therefore, we obtain the desired result.
        \qed
    \end{proof}

    Comparing the right hand side of~\eqref{ineq:L-bpg} and that of~\eqref{ineq:L-bpdc-s}, we can see that
    \begin{align}
        \label{ineq:l-smad-compare}
        3\left\|\sum_{r=1}^m \|a_r\|^2 a_r a_r^{\mathrm{T}}\right\|
        \leq \sum_{r=1}^m \left( 3\|a_r a_r^{\mathrm{T}}\|^2 + \|a_r a_r^{\mathrm{T}}\||b_r| \right).
    \end{align}
    The constant $L$ has the important role of defining the step size, and thereby affects the performance of the algorithms.
    Note that even if $\left\|\sum_{r=1}^m \|a_r\|^2 a_r a_r^{\mathrm{T}}\right\| = \sum_{r=1}^m \|a_r a_r^{\mathrm{T}}\|^2$,
    the left-hand side of~\eqref{ineq:l-smad-compare} is always smaller than
    the right-hand side because $\sum_{r=1}^m \|a_r a_r^{\mathrm{T}}\||b_r| \geq 0$.
    When $h(x)=\frac{1}{4}\|x\|^4+\frac{1}{2}\|x\|^2$, the subproblems of BPG(e) have a closed-form solution formula~\cite[Proposition 5.1]{bpg}.
    When $h(x) = \frac{1}{4}\|x\|^4$, subproblems~\eqref{subprob:bpdca} and~\eqref{subprob:bpdcae} also have a closed-form solution formula,
    which is obtained by slightly modifications of those in BPG(e).

    In this application, the functions $f_1, f_2, g$, and $h$ satisfy
    Assumptions from~\ref{ass1} to~\ref{ass6} excepting Assumption~\ref{ass4} (i) and~\ref{ass5g}.
    In particular, Assumption~\ref{ass4} (i) is not satisfied for our choice $h(x)=\frac{1}{4}\|x\|^4$,
    but it is satisfied if we replace it by $h(x)=\frac{1}{4}\|x\|^4+\frac{1}{2}\|x\|^2$.
    Finally, $\Psi$ and $H_M$ are KL functions due to their semi-algebraicity~\cite{attouch09}.
    Therefore, in this application, Assumption~\ref{ass5g} is not required for the global convergence of BPDCAe.

    \subsection{Lower Bound on the $L$-smooth Adaptable Parameter in the Gaussian Model}\label{subsec:the-lower-bound-for-l-smooth-adaptive}
    We dealt with the following Gaussian model.
    We generated the elements of $m$ vectors $a_r \in \real^d$ and
    the ground truth $\tilde{x} \in \real^d$,
    which was a sparse vector (sparsity of 5\%),
    independently from the standard Gaussian distribution.
    Then, we generated $b_r = \langle a_r, \tilde{x} \rangle^2, r = 1,2,\ldots, m$ from $a_r$ and $\tilde{x}$.

    From the linearity of the expectation, we consider the expectation of $\nabla^2 f_1$,
    \begin{align*}
        \mathbb{E}\left[ \nabla^2 f_1(x)\right]
        = 3\sum_{r=1}^m \mathbb{E}\left[\langle a_r, x \rangle^2 a_r a_r^{\mathrm{T}} \right].
    \end{align*}
    Since the elements of $a_r$ are independently generated from the standard Gaussian distribution,
    the $j$-th diagonal element of the above matrix is given by
    \begin{align*}
        \mathbb{E}\left[\langle a_r, x \rangle^2 a_{r, j}^2 \right]
        = \mathbb{E}\left[a_{r, j}^4 x_j^2 + \sum_{k=1, k\neq j}^d a_{r, j}^2 a_{r, k}^2 x_k^2 \right]
        = 3x_j^2 + \sum_{k=1, k\neq j}^d x_k^2
        = 2x_j^2 + \|x\|^2.
    \end{align*}
    The non-diagonal $(j, k)$ elements are
    \begin{align*}
        \mathbb{E}\left[\langle a_r, x \rangle^2 a_{r, j}a_{r, k} \right]
        &= \mathbb{E}\left[2a_{r, j}^2 a_{r, k}^2 x_j x_k \right] = 2 x_j x_k.
    \end{align*}
    Moreover, noting that $h(x) = \frac{1}{4}\|x\|^4$,
    we obtain $\mathbb{E}\left[\langle a_r, x \rangle^2 a_r a_r^{\mathrm{T}} \right] = \|x\|^2 I_d + 2xx^{\mathrm{T}} = \nabla^2 h(x)$.
    The expectation of the Hessian of $f_1(x)$ is thus given by $\mathbb{E}[\nabla^2 f_1(x)]=3m\nabla^2 h(x)$.

    Under the Gaussian model, we can reduce the lower bound of $L$ given in Proposition~\ref{prop:L-bpdc-sum} with high probability by applying~\cite[Lemma 7.4]{wirtinger} as shown in the following proposition.

    \begin{proposition}
        \label{prop:L-bpdc-f}
        Let the functions $f_1$ and $h$ be given by~\eqref{func:dc-real} and~\eqref{func:ker-real}, respectively.
        Moreover, assume that the vectors $a_r$ are independently distributed according to the Gaussian model with a sufficiently large number of measurements.
        Let $\gamma$ and $\delta$ be a fixed positive numerical constant and
        $c(\cdot)$ be a sufficiently large numerical constant that depends on $\delta$;
        this means that the number of samples obeys $m \geq c(\delta) \cdot d\log d$ in the Gaussian model.
        Then, for any L satisfying
        \begin{align}
            \label{ineq:L-bpdc-f}
            L \geq 9 \left\| \sum_{r=1}^m a_r a_r^{\mathrm{T}} \right\| + \delta,
        \end{align}
        the function $Lh-f_1$ is convex on $\real^d$
        and hence the pair $(f_1,h)$ is L-smad on $\real^d$
        with probability at least $1 - 5e^{-\gamma d} - 4 / d^2$.
    \end{proposition}

    \begin{proof}
        Consider the expectation of $\sum_{r=1}^m a_r a_r^{\mathrm{T}}$.
        Since the elements of $a_r$ are independently generated from the standard Gaussian distribution,
        for any $y\in\real^d$, we have
        \begin{align}
            y^{\mathrm{T}}\mathbb{E}\left[\sum_{r=1}^m a_r a_r^{\mathrm{T}}\right]y
            = \sum_{r=1}^m\mathbb{E}\left[\langle a_r, y \rangle^2\right]
            = \sum_{r=1}^m\sum_{j=1}^d y_j^2
            = \sum_{r=1}^m\|y\|^2. \label{ineq:mean-A}
        \end{align}
        From~\eqref{ineq:mean-A}, for any $y\in\real^d$, we have
        \begin{align}
            y^{\mathrm{T}}\mathbb{E}[\nabla^2 f_1(x)]y
            =3\sum_{r=1}^m\left( \|x\|^2\|y\|^2 + 2\langle x, y \rangle^2 \right)
            \leq 9\|x\|^2\sum_{r=1}^m\|y\|^2
            = 9\|x\|^2 y^{\mathrm{T}}\mathbb{E}\left[\sum_{r=1}^m a_r a_r^{\mathrm{T}}\right]y. \label{ineq:hessian-f1}
        \end{align}
        We can easily find that
        \begin{align*}
            9\sum_{r=1}^m a_r a_r^{\mathrm{T}}
            \preceq 9\left\|\sum_{r=1}^m a_r a_r^{\mathrm{T}}\right\| I_d,
        \end{align*}
        which implies that
        \begin{align}
            9\mathbb{E}\left[\sum_{r=1}^m a_r a_r^{\mathrm{T}}\right]
            \preceq 9\left\|\sum_{r=1}^m a_r a_r^{\mathrm{T}}\right\| I_d. \label{l_ineq}
        \end{align}
        From~\eqref{ineq:hessian-f1} and~\eqref{l_ineq},
        we have
        \begin{align}
            \mathbb{E}[\nabla^2 f_1(x)]
            \preceq 9\|x\|^2\left\|\sum_{r=1}^m a_r a_r^{\mathrm{T}}\right\| I_d \label{ineq:hessian-f1-L}.
        \end{align}
        From~\cite[Lemma 7.4]{wirtinger},~\eqref{ineq:L-bpdc-f}, and~\eqref{ineq:hessian-f1-L},
        we conclude that
        \begin{align}
            \label{ineq:f_1-L}
            \nabla^2 f_1(x)
            \preceq \mathbb{E}[\nabla^2 f_1(x)] + \delta\|x\|^2 I_d
            \preceq L\|x\|^2 I_d
        \end{align}
        with probability at least $1 - 5e^{-\gamma d} - 4 / d^2$.
        From $\nabla^2 h(x) \succeq \|x\|^2 I_d$ and~\eqref{ineq:f_1-L}, we have $\nabla^2 f_1(x)\preceq L\nabla^2 h(x)$,
        which proves that $Lh-f_1$ is convex with probability at least $1 - 5e^{-\gamma d} - 4 / d^2$.
        Therefore, the pair $(f_1,h)$ is $L$-smad on $\real^d$.
        \qed
    \end{proof}
    \begin{remark}
        Since each element of $a_r$ independently follows the standard Gaussian distribution,
        $\|a_r\|^2$ follows the chi-squared distribution with $d$ degrees of freedom.
        Thus, we can show $\|a_r\|^2 \geq 3$ with high probability for sufficiently large $d$.
        It implies that the bound given in Proposition~\ref{prop:L-bpdc-f} is smaller than that given in Proposition~\ref{prop:L-bpdc-sum}.
    \end{remark}

    \subsection{Performance Results for Phase Retrieval with the Gaussian Model}\label{subsec:the-performance-results}
    Here, we summarize the results for the Gaussian model.
    All numerical experiments were performed in Python 3.7 on an iMac with a 3.3 GHz Intel Core i5 Processor and
    8 GB 1867 MHz DDR3 memory.

    First, let us examine the results for Bregman proximal-type algorithms,
    \ie, BPG \cite{bpg}, BPGe~\cite{bpge}, BPDCA (Algorithm 1), and BPDCAe (Algorithm 2).
    We compared the averages of 100 random instances in terms of the number of iterations,
    CPU time, and accuracy (Tables~\ref{tab:non-ex} and~\ref{tab:ex}).
    Let $\hat{x}$ be a recovered solution and
    $\tilde{x}$ be the ground truth generated according to the method described in Subsection~\ref{subsec:the-lower-bound-for-l-smooth-adaptive}.
    In order to compare the objective function values,
    we took the difference $\log_{10}|\Psi(\hat{x}) - \Psi(\tilde{x})|$ to be the accuracy.
    In the numerical experiments, $\Psi(\hat{x}) > \Psi(\tilde{x})$.
    The termination criterion was defined as $\|x^k - x^{k-1}\|/\max\{1, \|x^k\|\} \leq 10^{-6}$.
    The equation numbers under each algorithm in Tables~\ref{tab:non-ex} and~\ref{tab:ex} indicate the value of $\lambda$;
    that is, we set $\lambda = 1/L$ for $L$ satisfying the equations.
    For the restart schemes, we used the adaptive restart scheme with $\rho = 0.99$ and
    the fixed restart scheme with $K=200$.
    We set $\theta = 1$ for the regularizer $g$ in~\eqref{prob:pr}.
    We forcibly stopped the algorithms when they reached the maximum number of iterations (50,000).
    Table~\ref{tab:ex} compares the results of BPGe and BPDCAe under the same settings as the results in Table~\ref{tab:non-ex}.
    BPDCA with~\eqref{ineq:L-bpdc-f} was the fastest among the algorithms without extrapolation (Table~\ref{tab:non-ex}).
    On the other hand, the extrapolation method makes each algorithm faster (Table~\ref{tab:ex}).

    \begin{table}[!htbp]
        \centering
        \caption[Performance tests comparing BPG~\cite{bpg} and BPDCA]{
            Average number of iterations, CPU time,
            and accuracy for BPG~\cite{bpg} and BPDCA using 100 random instances of phase retrieval
            (over the Gaussian model) for different values of $L$.
        }
        \label{tab:non-ex}
        \begin{tabular}{|c|r|r|r|r|r|}
            \hline
            Algorithm               & $m$    & $d$ & Iteration & CPU-Time (s) & Accuracy \\ \hline \hline
            BPG~\cite{bpg}                     & 10,000 & 10  & 3,757     & 1.638  & 2.901    \\ \cline{3-6}
            ~\eqref{ineq:L-bpg}     &        & 50  & 50,000    & 37.761 & 1.977    \\ \cline{3-6}
            &        & 100 & 50,000    & 46.920 & 5.312    \\ \cline{3-6}
            &        & 200 & 50,000    & 91.925 & 7.737    \\ \cline{2-6}
            & 20,000 & 10  & 3,689     & 2.539  & $-$2.569 \\ \cline{3-6}
            &        & 50  & 50,000    & 76.020& 2.007    \\ \cline{3-6}
            &        & 100 & 50,000    & 121.966& 5.523    \\ \cline{3-6}
            &        & 200 & 50,000    & 191.780& 8.057    \\ \cline{2-6}
            & 30,000 & 10  & 3,764     & 3.698  & $-$2.387 \\ \cline{3-6}
            &        & 50  & 50,000    & 104.947& 2.257    \\ \cline{3-6}
            &        & 100 & 50,000    & 175.143& 5.678    \\ \cline{3-6}
            &        & 200 & 50,000    & 287.735& 8.227    \\ \cline{1-6}
            BPDCA                   & 10,000 & 10  & 265       & 0.102  & $-$4.374 \\ \cline{3-6}
            ~\eqref{ineq:L-bpdc-s}  &        & 50  & 1,415	   & 0.520  & $-$3.212 \\ \cline{3-6}
            &        & 100 & 3,274	   & 2.129  & $-$2.656 \\ \cline{3-6}
            &        & 200 & 8,111	   & 10.416 & $-$2.061 \\ \cline{2-6}
            & 20,000 & 10  & 255	   & 0.157  & $-$4.350 \\ \cline{3-6}
            &        & 50  & 1,299	   & 1.182  & $-$3.193 \\ \cline{3-6}
            &        & 100 & 2,833	   & 4.283  & $-$2.642 \\ \cline{3-6}
            &        & 200 & 6,572	   & 18.198 & $-$2.057 \\ \cline{2-6}
            & 30,000 & 10  & 256	   & 0.233  & $-$4.335 \\ \cline{3-6}
            &        & 50  & 1,257 	   & 1.790  & $-$3.156 \\ \cline{3-6}
            &        & 100 & 2,696	   & 6.484  & $-$2.596 \\ \cline{3-6}
            &        & 200 & 6,012     & 25.666 & $-$2.010 \\ \cline{1-6}
            BPDCA                   & 10,000 & 10  & 68	       & 0.025 & $-$5.127 \\ \cline{3-6}
            ~\eqref{ineq:L-bpdc-f}  &        & 50  & 92	       & 0.034 & $-$4.627 \\ \cline{3-6}
            &        & 100 & 115	   & 0.075 & $-$4.380 \\ \cline{3-6}
            &        & 200 & 152	   & 0.192  & $-$4.108 \\ \cline{2-6}
            & 20,000 & 10  & 65	       & 0.040 & $-$5.137 \\ \cline{3-6}
            &        & 50  & 84	       & 0.077 & $-$4.691 \\ \cline{3-6}
            &        & 100 & 98	       & 0.149  & $-$4.476 \\ \cline{3-6}
            &        & 200 & 121	   & 0.335  & $-$4.229 \\ \cline{2-6}
            & 30,000 & 10  & 65	       & 0.059 & $-$5.166 \\ \cline{3-6}
            &        & 50  & 81	       & 0.115  & $-$4.728 \\ \cline{3-6}
            &        & 100 & 93	       & 0.223  & $-$4.515 \\ \cline{3-6}
            &        & 200 & 110	   & 0.465  & $-$4.285 \\ \cline{1-6}
        \end{tabular}
    \end{table}
    \begin{table}
        \centering
        \caption[Performance tests comparing BPGe~\cite{bpge} and BPDCAe]{
            Average number of iterations, CPU time,
            and accuracy for BPGe~\cite{bpge} and BPDCAe using 100 random instances of phase retrieval
            (over the Gaussian model) for different values of $L$.
        }
        \label{tab:ex}
        \begin{tabular}{|c|r|r|r|r|r|}
            \hline
            Algorithm               & $m$   & $d$ & Iteration & CPU-Time (s)& Accuracy \\ \hline \hline
            BPGe~\cite{bpge}                    & 10,000 & 10  & 297	  & 0.124 & $-$3.904 \\ \cline{3-6}
            ~\eqref{ineq:L-bpg}     &        & 50  & 2,614	  & 1.209 & $-$0.428 \\ \cline{3-6}
            &        & 100 & 6,214    &	5.949 & 0.974    \\ \cline{3-6}
            &        & 200 & 23,940   & 44.218& 2.426    \\ \cline{2-6}
            & 20,000 & 10  & 285	  & 0.198 & $-$3.653 \\ \cline{3-6}
            &        & 50  & 1,941    & 2.871 & $-$0.375 \\ \cline{3-6}
            &        & 100 & 6,054    & 15.376& 1.250    \\ \cline{3-6}
            &        & 200 & 21,138   & 82.086& 2.734    \\ \cline{2-6}
            & 30,000 & 10  & 294	  & 0.290 & $-$3.362 \\ \cline{3-6}
            &        & 50  & 1,880    &	3.826 & $-$0.199 \\ \cline{3-6}
            &        & 100 & 6,002    &	21.271& 1.411    \\ \cline{3-6}
            &        & 200 & 21,434   &	123.504& 2.806    \\ \cline{1-6}
            BPDCAe                  & 10,000 & 10  & 67	      & 0.025 & $-$5.205 \\ \cline{3-6}
            ~\eqref{ineq:L-bpdc-s}  &        & 50  & 203	  & 0.075 & $-$3.802 \\ \cline{3-6}
            &        & 100 & 332      & 0.218 & $-$3.451 \\ \cline{3-6}
            &        & 200 & 581      & 0.740 & $-$2.941 \\ \cline{2-6}
            & 20,000 & 10  & 62	      & 0.038 & $-$5.071 \\ \cline{3-6}
            &        & 50  & 179	  & 0.165 & $-$4.152 \\ \cline{3-6}
            &        & 100 & 302      & 0.458 & $-$3.694 \\ \cline{3-6}
            &        & 200 & 501	  & 1.394 & $-$3.110 \\ \cline{2-6}
            & 30,000 & 10  & 59	      & 0.054 & $-$4.852 \\ \cline{3-6}
            &        & 50  & 169	  & 0.242 & $-$4.054 \\ \cline{3-6}
            &        & 100 & 278	  & 0.670 & $-$3.448 \\ \cline{3-6}
            &        & 200 & 446	  & 1.891 & $-$2.987 \\ \cline{1-6}
            BPDCAe                  & 10,000 & 10  & 32	      & 0.013 & $-$5.649 \\ \cline{3-6}
            ~\eqref{ineq:L-bpdc-f}  &        & 50  & 42	      & 0.015 & $-$5.371 \\ \cline{3-6}
            &        & 100 & 49	      & 0.032 & $-$5.087 \\ \cline{3-6}
            &        & 200 & 61	      & 0.078 & $-$5.135 \\ \cline{2-6}
            & 20,000 & 10  & 29	      & 0.018 & $-$5.550 \\ \cline{3-6}
            &        & 50  & 38	      & 0.035 & $-$5.317 \\ \cline{3-6}
            &        & 100 & 43	      & 0.065 & $-$4.919 \\ \cline{3-6}
            &        & 200 & 52	      & 0.144 & $-$5.051 \\ \cline{2-6}
            & 30,000 & 10  & 29	      & 0.026 & $-$5.558 \\ \cline{3-6}
            &        & 50  & 38	      & 0.056 & $-$5.446 \\ \cline{3-6}
            &        & 100 & 41	      & 0.098 & $-$4.908 \\ \cline{3-6}
            &        & 200 & 50	      & 0.210 & $-$5.115 \\ \cline{1-6}
        \end{tabular}
    \end{table}
    \begin{figure}[!htbp]
        \begin{center}
            \includegraphics[width=0.7\textwidth]{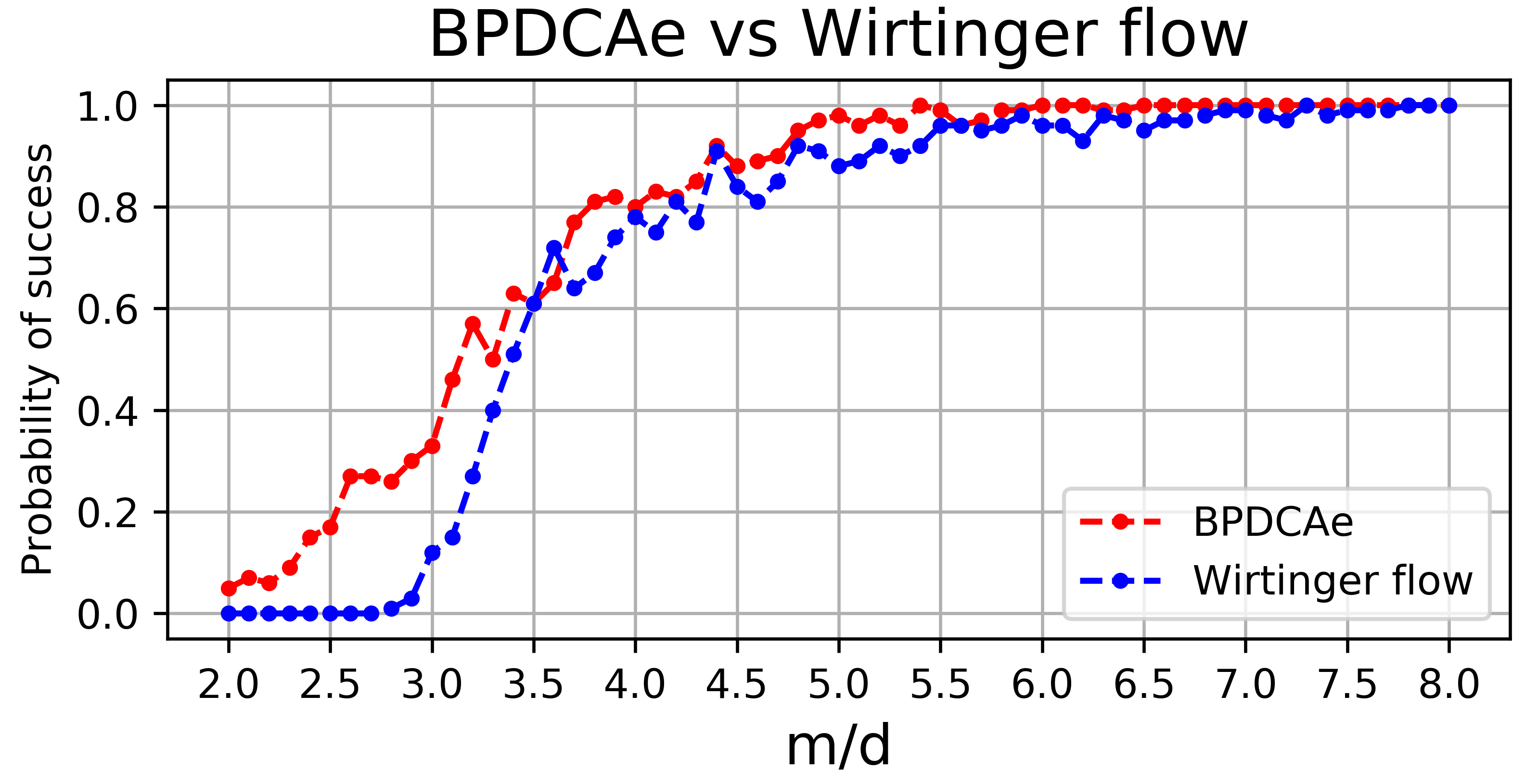}
        \end{center}
        \caption{Empirical probability of success
        based on 100 trials for BPDCAe and Wirtinger flow~\cite{wirtinger}
        using the same initialization step (of the Wirtinger flow).
        We set $d = 128$ and varied the number $m$ of measurements.}
        \label{fig:prob}
    \end{figure}

    We can conclude that, at least for phase retrieval,
    BPDCA has a clear advantage over BPG
    because of its reformulation as a nonconvex DC optimization problem~\eqref{func:dc-real},
    which permits choosing a smaller $L$ in~\eqref{ineq:L-bpdc-s} instead of~\eqref{ineq:L-bpg}.
    In particular, for the Gaussian model,
    we can use a smaller $L$ in ~\eqref{ineq:L-bpdc-f} with high probability.
    The extrapolation technique can further enhance performance.
    Also, we can see that the iterates of BPDCA(e) globally converge to their optimal solutions
    despite that the kernel generating distance $h$~\eqref{func:ker-real} does not satisfy Assumption~\ref{ass4} (i).
    This suggests that this condition may be relaxed in some cases.

    Next, we compared the empirical probability of success for BPDCAe and Wirtinger flow~\cite{wirtinger},
    which is a well-known algorithm for phase retrieval.
    Here we took $x_0$ in BPDCAe to be the value calculated in the initialization step of the Wirtinger flow.
    The empirical probability of success in Fig.~\ref{fig:prob} is an average over 100 trials.
    We regard that the algorithms succeeded if the relative error $\|\hat{x} - \tilde{x}\|/\|\tilde{x}\|$
    falls below $10^{-5}$ after 2,500 iterations.
    The dimension $d$ was fixed at 128, and we varied the number of measurements $m$.
    We used the adaptive restart scheme with $\rho = 0.99$ and the fixed restart scheme with $K=200$.
    We set $\theta = 0$; \ie, we solved~\eqref{prob:pr} without its regularizer.
    From the figure, we can see that BPDCAe with the initialization step of the Wirtinger flow
    achieved almost 100\% success rate when $m/d \geq 6$ and obtained more stable results than those of Wirtinger flow.

    \section{Conclusions}\label{sec:conclusions}
    We proposed two Bregman proximal-type algorithms for solving DC optimization problems $(\mathcal{P})$.
    One is the Bregman Proximal DC Algorithm (BPDCA), the other is BPDCA with extrapolation (BPDCAe).
    Proximal-type algorithms including ours are effective on large-scale problems.
    In addition, our algorithms assume
    that the function $f_1$ has the $L$-smooth adaptable property
    in relation to the kernel generating distance $h$, instead of $L$-smoothness.
    The restart condition for our adaptive restart scheme is different from the existing ones.

    We conducted convergence analyses of our algorithms.
    Assuming the Kurdyka-\L ojasiewicz property or subanalyticity of the objective function
    together with some standard assumptions,
    we established that the iterates generated by BPDCA(e) globally converge to a limiting stationary point or a limiting critical point
    and derived their convergence rates.

    We applied our algorithms to phase retrieval.
    The numerical experiments demonstrated that BPDCAe is faster than the other Bregman-type algorithms.
    For the Gaussian model, BPDCAe offered more stable results than Wirtinger flow~\cite{wirtinger}.
    We conclude that BPDCAe is a powerful method for solving large-scale and structured DC optimization problems.
    Although the kernel generating distance $h$~\eqref{func:ker-real} does not satisfy Assumption~\ref{ass4} (i),
    the sequences generated by BPDCA(e) converged in the numerical experiments.
    Therefore, we conjecture that most of the convergent results can be demonstrated under weaker conditions.
    As future work, since $g$ in BPDCA does not need to be convex,
    we will attempt to prove the monotonicity of the auxiliary function of BPDCAe (Lemma~\ref{lemma:obj-decrease-ex})
    without assuming Assumption~\ref{ass:g-convex}.

    Other Bregman proximal-type algorithms have been proposed.
    Mukkamala \etal~\cite{Mukkamala2020} chose the $L$-smad parameters by using a line search.
    As this parameter is generally difficult to estimate accurately, we can utilize this line search in our algorithms.

    For constrained problems,
    Wang \etal~\cite{NIPS2014_5612} proposed the Bregman alternating direction methods with multipliers.
    Tu \etal~\cite{Tu2020} also developed a Bregman-type algorithm for solving linearly constrained DC optimization problems.
    These variational methods may inspire further improvements and extensions.

    \section*{Acknowledgements}
    M. F. was supported by a JSPS KAKENHI Grant Number
    JP18K11178, from the Japan Society for the Promotion of Science (JSPS) and
    grants 2020/04585-7 and 2018/24293-0 from the S\~ao Paulo Research Foundation (FAPESP).
    M. T. was supported by a JSPS KAKENHI Grant Number
    JP19K15247, from the Japan Society for the Promotion of Science (JSPS).

    \section*{Data availability statement}
    The datasets generated during and/or analysed during the current study are available in the Github repository,
    https://github.com/ShotaTakahashi/bregman-proximal-dc-algorithm.

    \bibliographystyle{plain}
    \bibliography{main}

    \setcounter{section}{0}
    \renewcommand{\thesection}{\Alph{section}}

    \section{Appendix: Proof of Convergence Theorems for BPDCAe}\label{sec:appendix}
    \subsection{Proof of Theorem~\ref{theorem:global-subsequential-conv-bpdcae}}\label{subsec:proof-ofcref}
    (i)
    Since $H_M(x^k, x^{k-1}) \leq H_M(x^0,x^{-1})$ for all $k \in \mathbb{N}$
    from Proposition~\ref{prop:bpdcae-property} (i),
    with $x^0 = x^{-1}$, we obtain
    \begin{align*}
        \Psi(x^k) \leq \Psi(x^k) + MD_h(x^{k-1}, x^k) = H_M(x^k, x^{k-1}) \leq H_M(x^0,x^{-1}) = \Psi(x^0),
    \end{align*}
    which shows that $\{x^k\}_{k=0}^{\infty}$ is bounded
    due to Assumption~\ref{ass4} (iii).

    (ii)
    From~\eqref{ineq:aux-decrease-ex}, we obtain
    \begin{align*}
        H_M(x^k, x^{k-1}) - H_M(x^{k+1}, x^{k})
        \geq{}& \left(\frac{1}{\lambda} - M \right)D_h(x^k, x^{k+1})
        + \left( M - \frac{\rho}{\lambda} \right) D_h(x^{k-1}, x^k)\\
        &+\left(\frac{1}{\lambda}-L\right)D_h(x^{k+1},y^k)\\
        \geq{}&\frac{\sigma(1-\lambda L)}{2\lambda}\left(\|x^{k+1}-x^k\|^2-\beta_k\|x^k-x^{k-1}\|^2\right),
    \end{align*}
    where the last inequality holds because $h$ is a $\sigma$-strongly convex function and
    the first two terms are nonnegative.
    Summing the above inequality from $k=0$ to $\infty$, we obtain
    \begin{align*}
        &\frac{\sigma(1-\lambda L)}{2\lambda}\left(\sum_{k=0}^{\infty}(1-\beta_{k+1})\|x^{k+1}-x^k\|^2-\beta_1\|x^0-x^1\|^2\right)\\
        \leq{}& H_M(x^0, x^{-1}) - \liminf_{n\to\infty}H_M(x^{n+1}, x^{n})\\
        ={}&\Psi(x^0) - \liminf_{n\to\infty}\left( \Psi(x^{n+1}) + MD_h(x^n, x^{n+1}) \right)\\
        \leq{}&\Psi(x^0) - v(\mathcal{P}) < \infty,
    \end{align*}
    which shows that $\lim_{k\to\infty}\|x^{k+1} - x^k\| = 0$ due to $\frac{1}{\lambda}-L>0$ and $\sup_{k>0}\beta_k<1$.

    (iii)
    Let $\tilde{x}$ be an accumulation point of $\{x^k\}_{k=0}^{\infty}$
    and let $\{x^{k_j}\}$ be a subsequence such that $\lim_{j\to\infty}x^{k_j} = \tilde{x}$.
    Then, from the first-order optimality condition of subproblem~\eqref{subprob:bpdcae}
    under Assumption~\ref{ass2-ex}, we have
    \begin{align*}
        0 \in \partial_{\mathrm{c}} g(x^{k_j+1}) + \nabla f_1(y^{k_j}) - \xi^{k_j} + \frac{1}{\lambda}\left(\nabla h(x^{k_j + 1}) - \nabla h(y^{k_j}) \right).
    \end{align*}
    Therefore,
    we obtain
    \begin{align}
        \label{con:1st-k_j-ex}
        \xi^{k_j} + \nabla f_1(x^{k_j+1}) - \nabla f_1(y^{k_j}) + \frac{1}{\lambda}\left(\nabla h(y^{k_j}) - \nabla h(x^{k_j + 1}) \right) \in \partial_{\mathrm{c}} g(x^{k_j+1}) + \nabla f_1(x^{k_j+1}).
    \end{align}
    From the boundedness of $\{x^{k_j}\}$ and the Lipschitz continuity of $\nabla h$ and $\nabla f_1$ on a bounded subset of $\real^d$,
    there exists $A_0 > 0$ such that
    \begin{align*}
        \left\| \nabla f_1(x^{k_j+1}) - \nabla f_1(y^{k_j}) + \frac{1}{\lambda}\left( \nabla h(y^{k_j}) - \nabla h(x^{k_j + 1})\right)\right\|
        &\leq A_0\| x^{k_j+1} - y^{k_j} \|.
    \end{align*}
    Therefore,
    using $\| x^{k_j+1} - x^{k_j} \|\to0$ and $\|x^{k_j} - x^{k_j-1}\|\to0$,
    we obtain
    \begin{align}
        \label{lim:1st-con-ex}
        \nabla f_1(x^{k_j+1}) - \nabla f_1(y^{k_j}) + \frac{1}{\lambda}\left( \nabla h(y^{k_j}) - \nabla h(x^{k_j + 1})\right) \to 0.
    \end{align}
    Note that the sequence $\{\xi^{k_j}\}$ is bounded as shown in Theorem~\ref{theorem:global-subsequential-conv-bpdca} (ii),
    and the sequence $\{x^{k_j}\}$ is bounded and converges to $\tilde{x}$.
    Thus, by taking the limit as $j\to\infty$ or more precisely, its subsequence,
    we can assume without loss of generality
    that $\lim_{j\to\infty}\xi^{k_j} =: \tilde{\xi}$ exists, which belongs to $\partial_{\mathrm{c}} f_2(\tilde{x})$
    since $f_2$ is continuous.
    Using this and~\eqref{lim:1st-con-ex},
    we take the limit of~\eqref{con:1st-k_j-ex}.
    Invoking $\|x^{k_j+1} - x^{k_j}\| \to 0$
    and the continuity of $g$ and $\nabla f_1$,
    we obtain
    $\tilde{\xi} \in \partial_{\mathrm{c}} g(\tilde{x})+\nabla f_1(\tilde{x})$.
    Therefore,
    $0 \in \partial_{\mathrm{c}} g(\tilde{x}) + \nabla f_1(\tilde{x}) - \partial_{\mathrm{c}} f_2(\tilde{x})$,
    which shows that $\tilde{x}$ is a limiting critical point of $(\mathcal{P})$.
    \qed

    \subsection{Proof of Proposition~\ref{prop:zeta-ex}}\label{subsec:proof-of-prop}
    (i)
    From Assumption~\ref{ass1} (v) and Proposition~\ref{prop:bpdcae-property} (i),
    the sequence $\{H_M(x^k, x^{k-1})\}_{k=0}^{\infty}$ is bounded from below and non-increasing.
    Consequently, using $\lim_{k\to\infty}D_h(x^{k-1}, x^k) = 0$ from Proposition~\ref{prop:bpdcae-property} (ii),
    we obtain
    $\lim_{k\to\infty}H_M(x^k, x^{k-1}) = \lim_{k\to\infty}\Psi(x^k)=:\zeta$.

    (ii)
    Take any $\hat{x}\in\Omega$, that is $\lim_{j\to\infty}x^{k_j}=\hat{x}$.
    From~\eqref{subprob:bpdcae}, it follows that
    \begin{align*}
        g(x^{k})&
        + \langle \nabla f_1(y^{k-1}) - \xi^{k-1}, x^{k} - y^{k-1}\rangle
        + \frac{1}{\lambda}D_h(x^{k}, y^{k-1})\\
        &{}\leq g(\hat{x}) + \langle \nabla f_1(y^{k-1}) - \xi^{k-1}, \hat{x} - y^{k-1}\rangle
        + \frac{1}{\lambda}D_h(\hat{x}, y^{k-1}).
    \end{align*}
    From the above inequality and the fact that $f_1$ is convex at $x^k$, we obtain
    \begin{align}
        g(x^{k}) + f_1(x^k) \leq{}& g(\hat{x})
        + \langle \nabla f_1(y^{k-1}) - \xi^{k-1}, \hat{x} - x^{k}\rangle
        + \frac{1}{\lambda}D_h(\hat{x}, y^{k-1}) - \frac{1}{\lambda}D_h(x^{k}, y^{k-1})\nonumber\\
        &+ f_1(\hat{x}) + \langle \nabla f_1(x^k), x^k-\hat{x}\rangle\nonumber\\
        \leq{}& g(\hat{x})
        + \langle \nabla f_1(y^{k-1}) - \xi^{k-1}, \hat{x} - x^{k}\rangle
        + \frac{1}{\lambda}D_h(\hat{x}, y^{k-1}) + \frac{1}{\lambda}D_h(y^{k-1},\hat{x})\nonumber\\
        &+ f_1(\hat{x}) + \langle \nabla f_1(x^k), x^k-\hat{x}\rangle,\label{ineq:apendix-sup-g}
    \end{align}
    where the second inequality comes from $-\frac{1}{\lambda}D_h(x^{k}, y^{k-1}) \leq 0$
    and $\frac{1}{\lambda} D_h(y^{k-1}, \hat{x}) \geq 0$.
    Since $\nabla h$ is continuous, we have
    \begin{align*}
        \lim_{j\to\infty} \left( D_h(\hat{x}, y^{k_j-1}) + D_h(y^{k_j-1}, \hat{x}) \right)
        \leq \lim_{j\to\infty}\|\nabla h(y^{k_j-1}) - \nabla h(\hat{x})\|\|y^{k_j-1} - \hat{x}\|
        = 0.
    \end{align*}
    Substituting $k_j$ for $k$ in~\eqref{ineq:apendix-sup-g} and limiting $j$ to $\infty$,
    we have, from Proposition~\ref{prop:bpdcae-property} (ii),
    \begin{align*}
        \limsup_{j\to\infty} \left( g(x^{k_j}) + f_1(x^{k_j}) \right)\leq g(\hat{x}) + f_1(\hat{x}),
    \end{align*}
    which provides $\limsup_{j\to\infty}\Psi(x^{k_j})\leq\Psi(\hat{x})$ from the continuity of $ - f_2$.
    Combining this and the lower semicontinuity of $\Psi$ yields $\Psi(x^{k_j})\to\Psi(\hat{x})=:\zeta$ as $j\to\infty$.
    Since $\hat{x}\in\Omega$ is arbitrary, we conclude that $\Psi \equiv \zeta$ on $\Omega$.\qed

    \subsection{Proof of Theorem~\ref{theorem:global-convergence-bpdcae}}\label{subsec:proof-ofcref2}
    (i)
    Let $\mu>0, k_0>0$, $\mathcal{N}_0$, and
    $\mathcal{N}:=\{ x\in\mathcal{N}_0\ |\ \dist(x, \Omega)<\mu\}$
    as defined in the proof of Theorem~\ref{theorem:global-convergence-bpdca} (i).

    We begin by considering the subdifferential of $H_M$ at $x^k$ for $k \geq k_0 + 1$, and obtain
    \begin{align}
        \label{partial-h}
        \partial H_M(x^k, x^{k-1}) = \nabla f_1(x^k) - \nabla f_2(x^k) + \partial_{\mathrm{c}} g(x^k) - M\partial (\nabla h(x^k)) (x^{k-1} - x^k).
    \end{align}
    Moreover, considering the first-order optimality condition of subproblem~\eqref{subprob:bpdcae},
    for any $k \geq k_0 + 1$, we have
    \begin{align*}
        \frac{1}{\lambda}\left(\nabla h(y^{k-1}) - \nabla h(x^k) \right) - \nabla f_1(y^{k-1}) + \nabla f_2(x^{k-1}) \in \partial_{\mathrm{c}} g(x^k),
    \end{align*}
    since $f_2$ is $\mathcal{C}^1$ on $\mathcal{N}$ and $x^{k-1}\in\mathcal{N}$ whenever $k \geq k_0 + 1$.
    Using the above relation and~\eqref{partial-h}, for a bounded $U^k \in \partial(\nabla h (x^k))$ which exists by Assumption~\ref{ass6}, we also obtain
    \begin{align*}
        \frac{1}{\lambda}\left(\nabla h(y^{k-1}) - \nabla h(x^k) \right)
        &+ \nabla f_1(x^k) - \nabla f_1(y^{k-1})\\
        &+ \nabla f_2(x^{k-1}) - \nabla f_2(x^k)
        + M U^k(x^k - x^{k-1}) \in \partial H_M(x^k, x^{k-1}).
    \end{align*}
    Due to the global Lipschitz continuity of $\nabla f_1, \nabla f_2$, and $\nabla h$ on $\mathcal{N}_0$,
    and the boundedness of $U^k$ from Assumption~\ref{ass6},
    we see that there exist $A_0 > 0$, $A_1 > 0$, and $A_2 > 0$ such that
    \begin{align*}
        \dist((0, 0), \partial H_M(x^k, x^{k-1})) &\leq A{_0}\|x^k - y^{k-1}\| + A_1\|x^k - x^{k-1}\|\nonumber\\
        &\leq A_2\left(\|x^k - x^{k-1}\| + \|x^{k-1} - x^{k-2}\|\right),
    \end{align*}
    where $k \geq k_0 + 1$.
    Since $\|x^k - x^{k-1}\| \to 0$ and $\|x^{k-1} - x^{k-2}\| \to 0$,
    we conclude the claim (i).

    (ii)
    Suppose that $\hat{x}\in\Omega$, $x^{k_j}\rightarrow \hat{x}$, and $x^{k_j-1}\rightarrow \hat{x}$ as in Proposition~\ref{prop:zeta-ex} (ii).
    Therefore, the set of accumulation points of $\{(x^k,x^{k-1})\}_{k=0}^{\infty}$ is $\Upsilon$.
    From Propositions~\ref{prop:bpdcae-property} and~\ref{prop:zeta-ex},
    \begin{align*}
        \lim_{k\to\infty} H_M(x^k,x^{k-1}) = \lim_{k\to\infty} \Psi(x^k)+M\lim_{k\to\infty} D_h(x^{k-1},x^k) = \zeta.
    \end{align*}
    Additionally, from Proposition~\ref{prop:zeta-ex} (ii), for any $(\hat{x}, \hat{x}) \in \Upsilon, \hat{x}\in\Omega$,
    we have $H_M(\hat{x}, \hat{x}) = \Psi(\hat{x}) = \zeta$.
    Since $\hat{x}$ is arbitrary, we conclude that $H_M \equiv \zeta$ on $\Upsilon$.

    (iii)
    The proof is similar to Theorem~\ref{theorem:global-convergence-bpdca} (ii).\qed

    \subsection{Proof of Theorem~\ref{theorem:global-convergence-bpdcae-g}}\label{subsec:proof-ofcref3}
    Let $k_1, \kappa_i, \nu_i$, and $\theta_i$ be defined similarly to the proof of Theorem~\ref{theorem:global-convergence-bpdca-g}.
    Using the differentiability of $g$ and~\cite[Theorem 3.1]{Bolte2007}\nocite{Bolte2007},
    we have
    \begin{align}
        &\left\|\nabla g(x^k) - \nabla g(x^{k+1})\right\|\leq\kappa\|x^k-x^{k+1}\|,\label{ineq:ge-xk} \\
        &|H_M(x^k, x^{k-1}) - \zeta|^{\theta} \leq \nu \|\hat{x}^k\|,\quad \hat{x}^k\in\partial(-H)(x^k, x^{k-1}), \quad \forall k \geq k_1+1,\label{ineq:he-xk}
    \end{align}
    where $\zeta = H_M(\tilde{x}, \tilde{x}) = \Psi(\tilde{x}), \tilde{x}\in\Omega$, $\kappa = \max_{j=1,\ldots,p}\kappa_i, \nu = \max_{j=1,\ldots,p}\nu_i$, and $\theta = \max_{j=1,\ldots,p}\theta_i$.
    From~\eqref{subprob:bpdcae}, we obtain
    \begin{align*}
        0 = \nabla g(x^{k+1}) + \nabla f_1(y^{k}) - \xi^{k} + \frac{1}{\lambda}\left(\nabla h(x^{k + 1}) - \nabla h(y^{k}) \right),
    \end{align*}
    which implies
    \begin{align*}
        &\nabla g(x^{k+1}) - \nabla g(x^k) + \nabla f_1(y^k) - \nabla f_1(x^k) + \frac{1}{\lambda}\left(\nabla h(x^{k + 1}) - \nabla h(y^{k}) \right) + MU^k(x^{k-1} - x^{k})\\
        ={}& \xi^k + MU^k(x^{k-1} - x^{k}) - \nabla f_1(x^{k}) - \nabla g(x^k)
        \in \partial(-H_M)(x^k, x^{k-1}),
    \end{align*}
    for some bounded $U^k \in \partial(\nabla h (x^k))$ and $\partial(-H_M)(x^k, x^{k-1}) = \partial_{\mathrm{c}} f_2(x^k) + M\partial(\nabla h(x^k))(x^{k-1} - x^{k}) - \nabla f_1(x^k) - \nabla g(x^k)$.
    Using~\eqref{ineq:ge-xk},~\eqref{ineq:he-xk},
    Assumption~\ref{ass4}, and
    the boundedness of $\partial (\nabla h(x^k))$ from Assumption~\ref{ass6},
    we obtain $C > 0$ such that
    \begin{align*}
        &|H_M(x^k, x^{k-1}) - \zeta|^{\theta}\\
        \leq{}& \nu\left\|\nabla g(x^{k+1}) - \nabla g(x^k) + \nabla f_1(y^k) - \nabla f_1(x^k) + \frac{1}{\lambda}\left(\nabla h(x^{k + 1}) - \nabla h(y^{k})\right) + MU^k(x^k - x^{k-1})\right\|\nonumber\\
        \leq{}& C(\|x^k-x^{k+1}\| + \|x^{k-1}-x^{k}\|), \quad \forall k \geq k_1+1,
    \end{align*}
    where the second inequality comes from $\nabla h(x^{k + 1}) - \nabla h(y^{k}) = \nabla h(x^{k + 1}) - \nabla h(x^k) + \nabla h(x^k) - \nabla h(y^{k})$.
    The rest of the proof is similar to Theorem~\ref{theorem:global-convergence-bpdca-g}
    \qed
\end{document}